%% file: lojasiewicz-simon_gradient_inequality_part1_harmonic.tex
\documentclass[11pt]{amsart}

\input{latexformat}

\input{latexnewmacros}
\numberwithin{equation}{section}

\usepackage{amscd}
\usepackage{amssymb}
\usepackage{graphicx}
\usepackage{hyperref}
\usepackage{mathrsfs}
\usepackage[usenames]{color}
\usepackage{paralist}
\usepackage{slashed}
\usepackage{url}
\usepackage{verbatim}
\usepackage{xr}

\hypersetup{pdftitle={{\L}ojasiewicz--Simon gradient inequalities for the harmonic map energy function}}
\hypersetup{pdfauthor={Paul M. N. Feehan and Manousos Maridakis}}

\begin{document}

\title[{\L}ojasiewicz--Simon gradient inequalities]{{\L}ojasiewicz--Simon gradient inequalities for the harmonic map energy function}

\author[Paul M. N. Feehan]{Paul M. N. Feehan}
\address{Department of Mathematics, Rutgers, The State University of New Jersey, 110 Frelinghuysen Road, Piscataway, NJ 08854-8019, United States of America}
\email{feehan@math.rutgers.edu}

\author[Manousos Maridakis]{Manousos Maridakis}
\address{Department of Mathematics, Rutgers, The State University of New Jersey, 110 Frelinghuysen Road, Piscataway, NJ 08854-8019, United States of America}
\email{mmaridaki1@gmail.com}

%COMMENT Remove hours and minutes for arXiv and journal versions and fix date
%\date{\today{ }\hhmm}
\date{This version: March 5, 2019}

\begin{abstract}
We apply our abstract gradient inequalities developed by the authors in \cite{Feehan_Maridakis_Lojasiewicz-Simon_Banach} to prove {\L}ojasiewicz--Simon gradient inequalities for the harmonic map energy function using Sobolev spaces which impose minimal regularity requirements on maps between closed, Riemannian manifolds. Our {\L}ojasiewicz--Simon gradient inequalities for the harmonic map energy function generalize those of Kwon \cite[Theorem 4.2]{KwonThesis}, Liu and Yang \cite[Lemma 3.3]{Liu_Yang_2010}, Simon \cite[Theorem 3]{Simon_1983}, \cite[Equation (4.27)]{Simon_1985}, and Topping \cite[Lemma 1]{Topping_1997}.
\end{abstract}

% AMS 2010 subject classifications (used in AMS journals)

\subjclass[2010]{Primary 58E20; secondary 37D15}

% AMS keywords (used in AMS journals)
\keywords{Harmonic maps, {\L}ojasiewicz--Simon gradient inequality, Morse--Bott theory on Banach manifolds}

% Acknowledge support
\thanks{Paul Feehan was partially supported by National Science Foundation grant DMS-1510064 and the Oswald Veblen Fund and Fund for Mathematics (Institute for Advanced Study, Princeton) during the preparation of this article.}

\maketitle
%TODO Remove for journal version
\tableofcontents

\section{Introduction}
\label{sec:Introduction}
Our primary goal in this article is to prove {\L}ojasiewicz--Simon gradient inequalities for the harmonic map energy function. A key feature of our results is that we use systems of Sobolev norms that appear to be as as weak as possible. In a wealth of applications, \emph{{\L}ojasiewicz--Simon gradient inequalities} have played a significant role in analyzing questions such as
\begin{inparaenum}[\itshape a\upshape)]
\item global existence, convergence, and analysis of singularities for solutions to nonlinear evolution equations that are realizable as gradient-like systems for an energy function,
\item uniqueness of tangent cones, and
\item energy gaps and discreteness of energies.
\end{inparaenum}
For applications of the {\L}ojasiewicz--Simon gradient inequality to the harmonic map energy function in particular, we refer to Irwin \cite{IrwinThesis}, Kwon \cite{KwonThesis}, Liu and Yang \cite{Liu_Yang_2010}, Simon \cite{Simon_1985}, and Topping \cite{ToppingThesis, Topping_1997}, while for a survey of applications to other energy functions, we refer the reader to Feehan and Maridakis \cite[Section 1]{Feehan_Maridakis_Lojasiewicz-Simon_Banach}.

There are essentially three approaches to establishing a {\L}ojasiewicz--Simon gradient inequality for a particular energy function arising in geometric analysis or mathematical physics:
\begin{inparaenum}[\itshape 1\upshape)]
\item establish the inequality from first principles,
\item adapt the argument employed by Simon in the proof of his
\cite[Theorem 3]{Simon_1983}, or
\item apply an abstract version of the {\L}ojasiewicz--Simon gradient inequality for an analytic or Morse--Bott function on a Banach space.
\end{inparaenum}
Most famously, the first approach is exactly that employed by Simon in \cite{Simon_1983}, although this is also the avenue followed by Kwon \cite{KwonThesis}, Liu and Yang \cite{Liu_Yang_2010} and Topping \cite{ToppingThesis, Topping_1997} for the harmonic map energy function. We establish versions of the {\L}ojasiewicz--Simon gradient inequality for the harmonic map energy function (Theorem
\ref{mainthm:Lojasiewicz-Simon_gradient_inequality_energy_functional_Riemannian_manifolds}), using systems of Sobolev norms in these applications that are (as best we can tell) as weak as possible. Our gradient inequality for the harmonic map energy function is a significant generalization of previous inequalities due to Kwon \cite[Theorem 4.2]{KwonThesis}, Liu and Yang \cite[Lemma 3.3]{Liu_Yang_2010}, Simon \cite[Theorem 3]{Simon_1983}, \cite[Equation (4.27)]{Simon_1985}, and Topping \cite[Lemma 1]{Topping_1997}.

We begin in Section \ref{subsec:Lojasiewicz-Simon_gradient_inequality_abstract_functional} by reviewing our abstract {\L}ojasiewicz--Simon gradient inequalities (Theorems \ref{mainthm:Lojasiewicz-Simon_gradient_inequality_dualspace}, \ref{mainthm:Lojasiewicz-Simon_gradient_inequality}, \ref{mainthm:Lojasiewicz-Simon_gradient_inequality2}, and \ref{mainthm:Optimal_Lojasiewicz-Simon_gradient_inequality_Morse-Bott_energy_functional}) for an analytic function on a Banach space. We state our results on {\L}ojasiewicz--Simon gradient inequalities for the harmonic map energy function in Section  \ref{subsec:Lojasiewicz-Simon_gradient_inequality_harmonic_map_functional}. In proving the main result,  Theorem \ref{mainthm:Lojasiewicz-Simon_gradient_inequality_energy_functional_Riemannian_manifolds},  one must exploit the Banach manifold structure of the space of $W^{k,p}$ maps.  While other authors have also considered the smooth manifold structure of spaces of maps between smooth manifolds (see Eichhorn \cite{Eichhorn_1993}, Krikorian \cite{Krikorian_1972}, or Piccione and Tausk \cite{Piccione_Tausk_2001}) or approximation properties (see Bethuel \cite{Bethuel_1991}), none appear to have considered the specific question of interest to us here, namely, the real analytic manifold structure of the space of Sobolev maps from a closed, Riemannian, $C^\infty$ manifold into a closed, real analytic, Riemannian manifold. Moreover, the question does not appear to be considered directly in standard references for harmonic maps (such as H{\'e}lein \cite{Helein_harmonic_maps}, Jost \cite{Jost_riemannian_geometry_geometric_analysis}, or Struwe \cite{Struwe_1996, Struwe_variational_methods}, or references cited therein). The main result, Theorem \ref{mainthm:Lojasiewicz-Simon_gradient_inequality_energy_functional_Riemannian_manifolds}, follows by applying our abstract {\L}ojasiewicz--Simon gradient inequality, Theorem \ref{mainthm:Lojasiewicz-Simon_gradient_inequality}, on the appropriate Banach charts.

\subsection{{\L}ojasiewicz--Simon gradient inequalities for analytic and Morse--Bott functions on Banach spaces}
\label{subsec:Lojasiewicz-Simon_gradient_inequality_abstract_functional}
We begin with two abstract versions of Simon's infinite-dimensional version \cite[Theorem 3]{Simon_1983} of the {\L}ojasiewicz gradient
inequality \cite{Lojasiewicz_1965}. Let $\sX$ be a Banach space and let $\sX^*$ denote its continuous dual space. We call a bilinear form\footnote{Unless stated otherwise, all Banach spaces are considered to be real in this article.}, $b:\sX\times\sX \to \RR$, \emph{definite} if $b(x,x) \neq 0$ for all $x \in \sX\less\{0\}$. We say that a continuous \emph{embedding} of a Banach space into its continuous dual space, $\jmath:\sX\to\sX^*$, is \emph{definite} if the pullback of the canonical pairing, $\sX\times\sX \ni (x,y) \mapsto \langle x,\jmath(y)\rangle_{\sX\times\sX^*} \to \RR$, is a definite bilinear form.

\begin{mainthm}[{\L}ojasiewicz--Simon gradient inequality for analytic functions on Banach spaces]
\label{mainthm:Lojasiewicz-Simon_gradient_inequality_dualspace}
(See Feehan and Maridakis \cite[Theorem 1]{Feehan_Maridakis_Lojasiewicz-Simon_Banach}.)  
Let $\sX \subset \sX^*$ be a continuous, definite embedding of a Banach space into its dual space. Let $\sU \subset \sX$ be an open subset, $\sE:\sU\to\RR$ be an analytic function, and $x_\infty\in\sU$ be a critical point of $\sE$, that is, $\sE'(x_\infty) = 0$. Assume that $\sE''(x_\infty):\sX\to \sX^*$ is a Fredholm operator with index zero. Then there are constants $Z \in (0, \infty)$, and $\sigma \in (0,1]$, and $\theta \in [1/2,1)$, with the following significance. If $x \in \sU$ obeys
\begin{equation}
\label{eq:Lojasiewicz-Simon_gradient_inequality_neighborhood}
\|x-x_\infty\|_\sX < \sigma,
\end{equation}
then
\begin{equation}
\label{eq:Lojasiewicz-Simon_gradient_inequality_analytic_functional}
\|\sE'(x)\|_{\sX^*} \geq Z|\sE(x) - \sE(x_\infty)|^\theta.
\end{equation}
\end{mainthm}

\begin{rmk}[Comments on the embedding hypothesis in Theorem \ref{mainthm:Lojasiewicz-Simon_gradient_inequality_dualspace} and comparison with Huang's Theorem]
\label{rmk:Embedding_hypothesis_Huang_theorem_2-4-5}
The hypothesis in Theorem \ref{mainthm:Lojasiewicz-Simon_gradient_inequality_dualspace} on definiteness of the continuous embedding, $\sX \subset \sX^*$, is easily achieved given a continuous and dense embedding $\sX$ into a Hilbert space $\sH$. See Feehan and Maridakis \cite[Remark 1.1]{Feehan_Maridakis_Lojasiewicz-Simon_Banach}. 
\end{rmk}

\begin{rmk}[Index of a Fredholm Hessian operator on a reflexive Banach space]
\label{rmk:Index_Fredholm_Hessian_operator_reflexive_Banach_space}
If $\sX$ is a reflexive Banach space in Theorem \ref{mainthm:Lojasiewicz-Simon_gradient_inequality_dualspace}, then the hypothesis that $\sE''(x_\infty):\sX\to \sX^*$ has index zero can be omitted, since $\sE''(x_\infty)$ is always a symmetric operator and thus necessarily has index zero when $\sX$ is reflexive by \cite[Lemma D.3]{Feehan_Maridakis_Lojasiewicz-Simon_harmonic_maps_v5}.
\end{rmk}

\begin{rmk}[Topping's {\L}ojasiewicz--Simon gradient inequality for maps from $S^2$ to $S^2$ with small energy]
\label{rmk:Topping_Lojasiewicz-Simon_gradient_inequality}
Since the energy function, $\sE:\sU \subset \sX \to \RR$, in Theorems \ref{mainthm:Lojasiewicz-Simon_gradient_inequality} or \ref{mainthm:Optimal_Lojasiewicz-Simon_gradient_inequality_Morse-Bott_energy_functional} often arises in applications in the context of Morse or Morse--Bott theory, it is of interest to know when the {\L}ojasiewicz--Simon neighborhood condition \eqref{eq:Lojasiewicz-Simon_gradient_inequality_neighborhood}, namely $\|x-x_\infty\|_\sX < \sigma$ for a point $x \in \sU$ and a critical point $x_\infty$ and small  $\sigma \in (0,1]$, can be relaxed to $|\sE(x) - \sE(x_\infty)| < \eps$ and small $\eps \in (0,1]$.

When $\sE$ is the harmonic map energy function for maps $f$ from $S^2$ to $S^2$, where $S^2$ has its standard round metric of radius one, Topping \cite[Lemma 1]{Topping_1997}
has proved a version of the {\L}ojasiewicz--Simon gradient inequality where the critical point $f_\infty$ is the constant map and $f$ is a smooth map that is only required to obey a small energy condition, $\sE(f) < \eps$, in order for the {\L}ojasiewicz--Simon gradient inequality \eqref{eq:Lojasiewicz-Simon_gradient_inequality_analytic_functional} to hold in the sense that $\|\sE'(f)\|_{L^2(S^2)} \geq Z|\sE(f)|^{1/2}$ for some constant $Z \in [1,\infty)$. An analogue of \cite[Lemma 1]{Topping_1997} may hold more generally for the harmonic map energy function in the case of maps $f$ from a closed Riemann surface $M$ into a closed Riemannian manifold $N$ such that $|\sE(f) - \sE(f_\infty)| < \eps$ for a small enough constant $\eps \in (0,1]$ and a harmonic map $f_\infty$ from $M$ to $N$.
\end{rmk}

As emphasized by one researcher, the hypotheses of Theorem \ref{mainthm:Lojasiewicz-Simon_gradient_inequality_dualspace} are restrictive. For example, even though its hypotheses allow $\sX$ to be a Banach space, when the Hessian, $\sE''(x_\infty)$, is defined by an elliptic, linear, second-order partial differential operator, then (in the notation of Remark \ref{rmk:Choice_Banach_and_Hilbert_spaces_Lojasiewicz-Simon_gradient_inequality}) one is naturally led to choose $\sX$ to be a Hilbert space, $W^{1,2}(M;V)$, with dual space, $\sX^* = W^{-1,2}(M;V^*)$, in order to obtain the required Fredholm property. However, such a choice could make it impossible to simultaneously obtain the required real analyticity of the function, $\sE:\sX \supset \sU \to \RR$. As explained in Remark \ref{rmk:Choice_Banach_and_Hilbert_spaces_Lojasiewicz-Simon_gradient_inequality}, the forthcoming generalization greatly relaxes these constraints and implies Theorem \ref{mainthm:Lojasiewicz-Simon_gradient_inequality_dualspace} as a corollary. We first recall the concept of a gradient map \cite[Section 2.1B]{Huang_2006}, \cite[Section 2.5]{Berger_1977}.

\begin{defn}[Gradient map]
\label{defn:Huang_2-1-1}
(See Huang \cite[Definition 2.1.1]{Huang_2006}.)
Let $\sU\subset \sX$ be an open subset of a Banach space, $\sX$, and let $\tilde\sX$ be a Banach space with continuous embedding, $\tilde\sX \subseteqq \sX^*$. A continuous map, $\sM:\sU\to \tilde\sX$, is called a \emph{gradient map} if there exists a $C^1$ function, $\sE:\sU\to\RR$, such that
\begin{equation}
\label{eq:Differential_and_gradient_maps}
\sE'(x)v = \langle v,\sM(x)\rangle_{\sX\times\sX^*}, \quad \forall\, x \in \sU, \quad v \in \sX,
\end{equation}
where $\langle \cdot , \cdot \rangle_{\sX\times\sX^*}$ is the canonical bilinear form on $\sX\times\sX^*$. The real-valued function, $\sE$, is called a \emph{potential} for the gradient map, $\sM$.
\end{defn}

When $\tilde\sX = \sX^*$ in Definition \ref{defn:Huang_2-1-1}, then the differential and gradient maps coincide.

\begin{mainthm}[Refined {\L}ojasiewicz--Simon gradient inequality for analytic functions on Banach spaces]
\label{mainthm:Lojasiewicz-Simon_gradient_inequality}
(See Feehan and Maridakis \cite[Theorem 2]{Feehan_Maridakis_Lojasiewicz-Simon_Banach}.)
Let $\sX$ and $\tilde\sX$ be Banach spaces with continuous embeddings,
%TODO - Where do we need $\sX \subset \tilde\sX?
$\sX \subset \tilde\sX \subset \sX^*$, and such that the embedding, $\sX \subset \sX^*$, is definite. Let $\sU \subset \sX$ be an open subset, $\sE:\sU\to\RR$ be a $C^2$ function with real analytic gradient map, $\sM:\sU\to\tilde\sX$, and $x_\infty\in\sU$ be a critical point of $\sE$, that is, $\sM(x_\infty) = 0$. If $\sM'(x_\infty):\sX\to \tilde\sX$ is a Fredholm operator with index zero, then there are constants, $Z \in (0,\infty)$, and $\sigma \in (0,1]$, and $\theta \in [1/2, 1)$, with the following significance. If $x \in \sU$ obeys
\begin{equation}
\label{eq:Lojasiewicz-Simon_gradient_inequality_neighborhood_general}
\|x-x_\infty\|_\sX < \sigma,
\end{equation}
then
\begin{equation}
\label{eq:Lojasiewicz-Simon_gradient_inequality_analytic_functional_general}
\|\sM(x)\|_{\tilde\sX} \geq Z|\sE(x) - \sE(x_\infty)|^\theta.
\end{equation}
\end{mainthm}

\begin{rmk}[Comments on the embedding hypothesis in Theorem \ref{mainthm:Lojasiewicz-Simon_gradient_inequality}]
\label{rmk:Embedding_hypothesis_Lojasiewicz-Simon_gradient_inequality}
The hypothesis in Theorem \ref{mainthm:Lojasiewicz-Simon_gradient_inequality} on the continuous embedding, $\sX \subset \sX^*$, is easily achieved given a continuous and dense embedding of $\sX$ into a Hilbert space $\sH$. See Feehan and Maridakis \cite[Remark 1.5]{Feehan_Maridakis_Lojasiewicz-Simon_Banach}.
\end{rmk}

\begin{rmk}[On the choice of Banach spaces in applications of Theorem \ref{mainthm:Lojasiewicz-Simon_gradient_inequality}]
\label{rmk:Choice_Banach_and_Hilbert_spaces_Lojasiewicz-Simon_gradient_inequality}
The hypotheses of Theorem \ref{mainthm:Lojasiewicz-Simon_gradient_inequality} are designed to give the most flexibility in applications of a {\L}ojasiewicz--Simon gradient inequality to analytic functions on Banach spaces. An example of a convenient choice of Banach spaces modeled as Sobolev spaces, when $\sM'(x_\infty)$ is realized as an elliptic partial differential operator of order $m$, would be
\[
\sX = W^{k,p}(X;V), \quad \tilde\sX = W^{k-m,p}(X;V), \quad\text{and}\quad \sX^* = W^{-k,p'}(X;V),
\]
where $k\in\ZZ$ is an integer, $p \in (1,\infty)$ is a constant with dual H\"older exponent $p'\in(1,\infty)$ defined by $1/p+1/p'=1$, while $X$ is a closed Riemannian manifold of dimension $d\geq 2$ and $V$ is a Riemannian vector bundle with a compatible connection, $\nabla:C^\infty(X;V) \to C^\infty(X;T^*X\otimes V)$, and $W^{k,p}(X;V)$ denotes a Sobolev space defined in the standard way \cite{Aubin_1998}. When the integer $k$ is chosen large enough, the verification of analyticity of the gradient map, $\sM:\sU\to\tilde\sX$, is straightforward. Normally, that is the case when $k\geq m+1$ and $(k-m)p>d$ or $k-m=d$ and $p=1$, since $W^{k-m,p}(X;\CC)$ is then a Banach algebra by \cite[Theorem 4.39]{AdamsFournier}. If the Banach spaces are instead modeled as H\"older spaces, as in Simon \cite{Simon_1983}, a convenient choice of Banach spaces would be
\[
\sX = C^{k,\alpha}(X;V) \quad \text{and} \quad \tilde\sX = C^{k-m,\alpha}(X;V),
\]
where $\alpha \in (0,1)$ and $k\geq m$, and these H\"older spaces are defined in the standard way \cite{Aubin_1998}. Following Remark~\ref{rmk:Embedding_hypothesis_Huang_theorem_2-4-5}, the definiteness of the embedding  $C^{k,\alpha}(X;V)=\sX \subset \sX^*$ in this case is achieved by observing that $C^{k,\alpha}(X;V) \subset L^2(X;V)$.
\end{rmk}
\
Theorem \ref{mainthm:Lojasiewicz-Simon_gradient_inequality} appears to us to be the most widely applicable abstract version of the {\L}ojasiewicz--Simon gradient inequality that we are aware of in the literature. However, for applications where $\sM'(x_\infty)$ is realized as an elliptic partial differential operator of \emph{even} order, $m=2n$, and the nonlinearity of the gradient map is sufficiently mild, it often suffices to choose $\sX$ to be the Banach space, $W^{n,2}(X;V)$, and choose $\tilde\sX = \sX^*$ to be the Banach space, $W^{-n,2}(X;V)$. The distinction between the differential, $\sE'(x) \in \sX^*$, and the gradient, $\sM(x) \in \tilde\sX$, then disappears. Similarly, the distinction between the Hessian, $\sE''(x_\infty) \in (\sX\times\sX)^*$, and the Hessian operator, $\sM'(x_\infty) \in \sL(\sX,\tilde\sX)$, disappears. Finally, if $\sE:\sX\supset\sU \to \RR$ is real analytic, then the simpler Theorem \ref{mainthm:Lojasiewicz-Simon_gradient_inequality_dualspace} is often adequate for applications.

While Theorem \ref{mainthm:Lojasiewicz-Simon_gradient_inequality} has important applications to proofs of global existence, convergence, convergence rates, and stability of gradient flows defined by an energy function, $\sE:\sX\supset \sU \to \RR$, with gradient map, $\sM:\sX\supset \sU \to \tilde\sX$, (see \cite[Section 2.1]{Feehan_yang_mills_gradient_flow_v4} for an introduction and Simon \cite{Simon_1983} for his pioneering development), the gradient inequality \eqref{eq:Lojasiewicz-Simon_gradient_inequality_analytic_functional_general} is most useful when it has the form,
\[
\|\sM(x)\|_{\sH} \geq Z|\sE(x) - \sE(x_\infty)|^\theta, \quad\forall\, x \in \sU \text{ with } \|x-x_\infty\|_\sX < \sigma,
\]
where $\sH$ is a Hilbert space and the Banach space, $\sX$, is a dense subspace of $\sH$ with continuous embedding, $\sX \subset \sH$, and so $\sH^* \subset \sX^*$ is also a continuous embedding.  For example, to obtain Theorem \ref{mainthm:Lojasiewicz-Simon_gradient_inequality_energy_functional_Riemannian_manifolds} for the harmonic map energy function, we choose
\[
\sX = W^{k,p}(M;f_\infty^*TN),
\]
but for applications to gradient flow, we would like to replace the gradient inequality \eqref{eq:Lojasiewicz-Simon_gradient_inequality_harmonic_map_energy_functional_Riemannian_manifold} by
\[
\|\sM(f)\|_{L^2(M;f^*TN)}
\geq
Z|\sE(f) - \sE(f_\infty)|^\theta,
\]
but under the original {\L}ojasiewicz--Simon neighborhood condition \eqref{eq:Lojasiewicz-Simon_gradient_inequality_harmonic_map_neighborhood_Riemannian_manifold},
\[
\|f - f_\infty\|_{W^{k,p}(M)} < \sigma.
\]
Unfortunately, such an $L^2$ gradient inequality
(or Simon's \cite[Theorem 3]{Simon_1983}, \cite[Equation (4.27)]{Simon_1985}) does not follow from Theorem \ref{mainthm:Lojasiewicz-Simon_gradient_inequality} when $M$ has dimension $d \geq 4$, as explained in the proof of Corollary \ref{maincor:Lojasiewicz-Simon_gradient_inequality_energy_functional_Riemannian_manifolds_L2} and Remark \ref{rmk:Lojasiewicz-Simon_gradient_inequality_energy_functional_manifolds_L2_exclude_d_geq_4}; see also \cite{Feehan_harmonic_map_relative_energy_gap}. However, these $L^2$ gradient inequalities \emph{are} implied by the forthcoming Theorem \ref{mainthm:Lojasiewicz-Simon_gradient_inequality2} which generalizes and simplifies Huang's \cite[Theorem 2.4.2 (i)]{Huang_2006} (see Feehan and Maridakis \cite[Theorem E.2]{Feehan_Maridakis_Lojasiewicz-Simon_coupled_Yang-Mills_v6}. We refer to Feehan and Maridakis \cite[Section 1.2]{Feehan_Maridakis_Lojasiewicz-Simon_Banach} for further discussion.

\begin{mainthm}[Generalized {\L}ojasiewicz--Simon gradient inequality for analytic functions on Banach spaces]
\label{mainthm:Lojasiewicz-Simon_gradient_inequality2}
(See Feehan and Maridakis \cite[Theorem 3]{Feehan_Maridakis_Lojasiewicz-Simon_Banach}.)
Let $\sX$ and $\tilde\sX$ be Banach spaces with continuous embeddings, $\sX \subset \tilde\sX \subset \sX^*$, and such that the embedding, $\sX \subset \sX^*$, is definite.  Let $\sU \subset \sX$ be an open subset, $\sE:\sU\to\RR$ be an analytic function, and $x_\infty\in\sU$ be a critical point of $\sE$, that is, $\sE'(x_\infty) = 0$. Let
\[
\sX\subset \sG \subset \tilde\sG \quad \text{and} \quad \tilde\sX \subset \tilde\sG \subset \sX^*,
\]
be continuous embeddings of Banach spaces such that the compositions,
\[
\sX\subset \sG\subset \tilde\sG \quad \text{and}\quad \sX\subset \tilde\sX\subset \tilde\sG,
\]
induce the same embedding, $\sX \subset \tilde\sG$. Let $\sM:\sU\to\tilde\sX$ be a gradient map for $\sE$ in the sense of Definition \ref{defn:Huang_2-1-1}. Suppose that for each $x \in \sU$, the bounded, linear operator,
\[
\sM'(x): \sX \to \tilde \sX,
\]
has an extension
\[
\sM_1(x): \sG \to \tilde\sG
\]
such that the map
\[
\sU \ni x \mapsto \sM_1(x) \in \sL(\sG, \tilde\sG) \quad\hbox{is continuous}.
\]
If $\sM'(x_\infty):\sX\to \tilde\sX$ and $\sM_1(x_\infty):\sG\to \tilde\sG$ are Fredholm operators with index zero, then there are constants, $Z \in (0,\infty)$ and $\sigma \in (0,1]$ and $\theta \in [1/2, 1)$, with the following significance. If $x \in \sU$ obeys
\begin{equation}
\label{eq:Lojasiewicz-Simon_gradient_inequality_neighborhood_general2}
\|x-x_\infty\|_\sX < \sigma,
\end{equation}
then
\begin{equation}
\label{eq:Lojasiewicz-Simon_gradient_inequality_analytic_functional_general2}
\|\sM(x)\|_{\tilde\sG} \geq Z|\sE(x) - \sE(x_\infty)|^\theta.
\end{equation}
\end{mainthm}

\begin{rmk}[Generalized {\L}ojasiewicz--Simon gradient inequality for analytic functions on Banach spaces with gradient map valued in a Hilbert space]
\label{rmk:Lojasiewicz-Simon_gradient_inequality2_Hilbert}
Suppose now that $\tilde\sG = \sH$, a Hilbert space, so that the embedding $\sG\subset \sH$  in Theorem \ref{mainthm:Lojasiewicz-Simon_gradient_inequality2}, factors through $\sG\subset \sH\simeq \sH^* $ and therefore
\[
\sE'(x)v = \langle v, \sM(x) \rangle_{\sX\times\sX^*} = (v, \sM(x))_\sH, \quad\forall\, x \in \sU \text{ and } v \in \sX,
\]
using the continuous embeddings, $\tilde\sX \subset \sH \subset \sX^*$. As we noted in Remark \ref{rmk:Embedding_hypothesis_Huang_theorem_2-4-5}, the hypothesis in Theorem \ref{mainthm:Lojasiewicz-Simon_gradient_inequality2} that the embedding, $\sX \subset \sX^*$, is definite is implied by the assumption that $\sX \subset \sH$ is a continuous embedding into a Hilbert space. By Theorem \ref{mainthm:Lojasiewicz-Simon_gradient_inequality2}, if $x \in \sU$ obeys
\begin{equation}
\label{eq:Lojasiewicz-Simon_gradient_inequality_neighborhood_general_Hilbert_space}
\|x-x_\infty\|_\sX < \sigma,
\end{equation}
then
\begin{equation}
\label{eq:Lojasiewicz-Simon_gradient_inequality_analytic_function_Hilbert_space}
\|\sM(x)\|_{\sH} \geq Z|\sE(x) - \sE(x_\infty)|^\theta,
\end{equation}
as desired.
\end{rmk}

\begin{rmk}
If the Banach spaces are instead modeled as H\"older spaces, as in Simon \cite{Simon_1983}, a convenient choice of Banach and Hilbert spaces would be
\[
\sX = C^{k,\alpha}(X;V), \quad \tilde\sX = C^{k-m,\alpha}(X;V), \quad\text{and}\quad \sH = L^2(X;V),
\]
where $\alpha \in (0,1)$ and $k\geq m$, and these H\"older spaces are defined in the standard way \cite{Aubin_1998}.
\end{rmk}

It is of considerable interest to know when the optimal exponent $\theta = 1/2$ is achieved, since in that case one can prove (see \cite[Theorem 24.21]{Feehan_yang_mills_gradient_flow_v4}, for example) that a global solution, $u:[0,\infty)\to\sX$, to a gradient system governed by the {\L}ojasiewicz--Simon gradient inequality,
\[
\frac{du}{dt} = -\sE'(u(t)), \quad u(0) = u_0,
\]
has \emph{exponential} rather than mere power-law rate of convergence
to the critical point, $u_\infty$. One simple version of such an
optimal {\L}ojasiewicz--Simon gradient inequality is provided in Huang
\cite[Proposition 2.7.1]{Huang_2006} which, although interesting, its
hypotheses are very restrictive, a special case of Theorem
\ref{mainthm:Lojasiewicz-Simon_gradient_inequality_dualspace} where $\sX$ is a Hilbert space and the Hessian, $\sE''(x_\infty):\sX\to \sX^*$, is an invertible operator. See Haraux, Jendoubi, and Kavian \cite[Proposition 1.1]{Haraux_Jendoubi_Kavian_2003} for a similar result.

For the harmonic map energy function, a more interesting optimal {\L}ojasiewicz--Simon-type gradient inequality,
\[
\|\sE'(f)\|_{L^p(S^2)} \geq Z|\sE(f) - \sE(f_\infty)|^{1/2},
\]
has been obtained by Kwon \cite[Theorem 4.2]{KwonThesis} for maps $f:S^2\to N$, where $N$ is a closed Riemannian manifold and $f$ is close to a harmonic map $f_\infty$ in the sense that
\[
\|f - f_\infty\|_{W^{2,p}(S^2)} < \sigma,
\]
where $p$ is restricted to the range $1 < p \leq 2$, and $f_\infty$ is assumed to be \emph{integrable} in the sense of \cite[Definitions 4.3 or 4.4 and Proposition 4.1]{KwonThesis}. Her \cite[Proposition 4.1]{KwonThesis} quotes results of Simon \cite[pp. 270--272]{Simon_1985} and Adams and Simon \cite{Adams_Simon_1988}.

The \cite[Lemma 3.3]{Liu_Yang_2010} due to Liu and Yang is another example of an optimal {\L}ojasiewicz--Simon-type gradient inequality for the harmonic map energy function, but restricted to the setting of maps $f:S^2\to N$, where $N$ is a K{\"a}hler manifold of complex dimension $n \geq 1$ and nonnegative bisectional curvature, and the energy $\sE(f)$ is sufficiently small. The result of Liu and Yang generalizes that of Topping \cite[Lemma 1]{Topping_1997}, who assumes that $N = S^2$.

For the Yamabe function, an optimal {\L}ojasiewicz--Simon gradient inequality, has been obtained by Carlotto, Chodosh, and Rubinstein \cite{Carlotto_Chodosh_Rubinstein_2015} under the hypothesis that the critical point is \emph{integrable} in the sense of their
\cite[Definition 8]{Carlotto_Chodosh_Rubinstein_2015}, a condition that they observe in \cite[Lemma 9]{Carlotto_Chodosh_Rubinstein_2015} (quoting \cite[Lemma 1]{Adams_Simon_1988} due to Adams and Simon) is equivalent to a function on Euclidean space given by the \emph{Lyapunov--Schmidt reduction} of $\sE$ being constant on an open neighborhood of the critical point.

For the Yang-Mills energy function for connections on a principal $U(n)$-bundle over a closed Riemann surface, an optimal {\L}ojasiewicz--Simon gradient inequality, has been obtained by R\r{a}de \cite[Proposition 7.2]{Rade_1992} when the Yang-Mills connection is \emph{irreducible}.

Given the desirability of treating an energy function as a \emph{Morse function} whenever possible, for example in the spirit of Atiyah and Bott \cite{Atiyah_Bott_1983} for the Yang-Mills equation over Riemann surfaces, it is useful to rephrase these integrability conditions in the spirit of Morse theory.

\begin{defn}[Morse--Bott function]
\label{defn:Morse-Bott_function}
(See Austin and Braam \cite[Section 3.1]{Austin_Braam_1995}.)
Let $\sB$ be a smooth Banach manifold, $\sE: \sB \to \RR$ be a $C^2$ function, and $\Crit\sE := \{x\in\sB:\sE'(x) = 0\}$. A smooth submanifold $\sC \hookrightarrow \sB$ is called a \emph{nondegenerate critical submanifold of $\sE$} if $\sC \subset \Crit\sE$ and
\begin{equation}
\label{eq:Nondegenerate_critical_submanifold}
(T\sC)_x = \Ker \sE''(x), \quad\forall\,x\in \sC,
\end{equation}
where $\sE''(x):(T\sB)_x \to (T\sB)_x^*$ is the Hessian of $\sE$ at the point $x \in \sC$. One calls $\sE$ a \emph{Morse--Bott function} if its critical set $\Crit\sE$ consists of nondegenerate critical submanifolds.

We say that a $C^2$ function $\sE:\sB\to\RR$ is \emph{Morse--Bott at a point $x_\infty \in \Crit\sE$} if there is an open neighborhood $\sU\subset\sB$ of $x_\infty$ such that $\sU\cap\Crit\sE$ is a relatively open, smooth submanifold of $\sB$ and \eqref{eq:Nondegenerate_critical_submanifold} holds at $x_\infty$.
\end{defn}

Definition \ref{defn:Morse-Bott_function} is a restatement of definitions of a Morse--Bott function on a finite-dimensional manifold, but we omit the condition that $\sC$ be compact and connected as in Nicolaescu \cite[Definition 2.41]{Nicolaescu_morse_theory} or the condition that $\sC$ be compact in Bott \cite[Definition, p. 248]{Bott_1954}. Given a Morse--Bott energy function, we have the

\begin{mainthm}[Optimal {\L}ojasiewicz--Simon gradient inequality for Morse--Bott functions on Banach spaces]
\label{mainthm:Optimal_Lojasiewicz-Simon_gradient_inequality_Morse-Bott_energy_functional}
(See Feehan and Maridakis \cite[Theorem 4]{Feehan_Maridakis_Lojasiewicz-Simon_Banach}.)  
Assume the hypotheses of Theorem \ref{mainthm:Lojasiewicz-Simon_gradient_inequality} or of Theorem \ref{mainthm:Lojasiewicz-Simon_gradient_inequality2}. 
%TODO Do we need index zero? Analytic?
If $\sM$ is $C^1$ and $\sE$ is a Morse--Bott function at $x_\infty$ in the sense of Definition \ref{defn:Morse-Bott_function}, then the conclusions of Theorem \ref{mainthm:Lojasiewicz-Simon_gradient_inequality} or \ref{mainthm:Lojasiewicz-Simon_gradient_inequality2} hold with $\theta = 1/2$.
\end{mainthm}

We refer to Feehan \cite[Appendix C]{Feehan_lojasiewicz_inequality_all_dimensions_morse-bott} for a discussion of integrability and the Morse--Bott condition for the harmonic map energy function, together with examples.

\subsection{{\L}ojasiewicz--Simon gradient inequality for the harmonic map energy function}
\label{subsec:Lojasiewicz-Simon_gradient_inequality_harmonic_map_functional}
Finally, we describe a consequence of Theorem \ref{mainthm:Lojasiewicz-Simon_gradient_inequality} for the harmonic map energy function. For background on harmonic maps, we refer to H{\'e}lein \cite{Helein_harmonic_maps}, Jost \cite{Jost_riemannian_geometry_geometric_analysis}, Simon \cite{Simon_1996}, Struwe \cite{Struwe_variational_methods}, and references cited therein. We begin with the

\begin{defn}[Harmonic map energy function]
\label{defn:Harmonic_map_energy_functional}
Let $(M,g)$ and $(N,h)$ be a pair of closed, smooth Riemannian manifolds. One defines the \emph{harmonic map energy function} by
\begin{equation}
\label{eq:Harmonic_map_energy_functional}
\sE_{g,h}(f)
:=
\frac{1}{2} \int_M |df|_{g,h}^2 \,d\vol_g,
\end{equation}
for smooth maps, $f:M\to N$, where $df:TM \to TN$ is the differential map.
\end{defn}

When clear from the context, we omit explicit mention of the Riemannian metrics $g$ on $M$ and $h$ on $N$ and write $\sE = \sE_{g,h}$. Although initially defined for smooth maps, the energy function $\sE$ in Definition \ref{defn:Harmonic_map_energy_functional}, extends to the case of Sobolev maps of class $W^{1,2}$. To define the gradient, $\sM = \sM_{g,h}$, of the energy function $\sE$ in \eqref{eq:Harmonic_map_energy_functional} with respect to the $L^2$ metric on $C^\infty(M;N)$, we first choose an isometric embedding, $(N,h) \hookrightarrow \RR^n$ for a sufficiently large $n$ (courtesy of the isometric embedding theorem due to Nash \cite{Nash_1956}), and recall that\footnote{Compare \cite[Equations (8.1.10) and (8.1.13)]{Jost_riemannian_geometry_geometric_analysis}, where Jost uses variations of $f$ of the form $\exp_{f}(tu)$.} by \cite[Equations (2.2)(i) and (ii)]{Simon_1996}
\begin{align*}
\left(u, \sM(f)\right)_{L^2(M,g)}
&:= \sE'(f)(u) = \left.\frac{d}{dt}\sE(\pi(f + tu))\right|_{t=0}
\\
&\,= \left(u, \Delta_g f\right)_{L^2(M,g)}
\\
&\,= \left(u, d\pi_h(f)\Delta_g f\right)_{L^2(M,g)},
\end{align*}
for all $u \in C^\infty(M;f^*TN)$, where $\pi_h$ is the nearest point projection onto $N$ from a normal tubular neighborhood and $d\pi_h(y):\RR^n \to T_yN$ is orthogonal projection, for all $y \in N$.  By \cite[Lemma 1.2.4]{Helein_harmonic_maps}, we have
\begin{equation}
\label{eq:Gradient_harmonic_map_operator}
\sM(f) = d\pi_h(f)\Delta_g f = \Delta_g f - A_h(f)(df,df),
\end{equation}
as in \cite[Equations (2.2)(iii) and (iv)]{Simon_1996}. Here, $A_h$ denotes the second fundamental form of the isometric embedding, $(N,h) \subset \RR^n$ and
\begin{equation}
\label{eq:Laplace-Beltrami_operator}
\Delta_g
:=
-\divg_g \grad_g
=
d^{*,g}d
=
-\frac{1}{\sqrt{\det g}} \frac{\partial}{\partial x^\beta}
\left(\sqrt{\det g}\, \frac{\partial f}{\partial x^\alpha} \right)
\end{equation}
denotes the Laplace-Beltrami operator for $(M,g)$ (with the opposite sign convention to that of \cite[Equations (1.14) and (1.33)]{Chavel}) acting on the scalar components $f^i$ of $f = (f^1,\ldots,f^n)$ and $\{x^\alpha\}$ denote local coordinates on $M$.

Given a smooth map $f:M\to N$, an isometric embedding, $(N,h) \subset \RR^n$, a non-negative integer $k$, and $p \in [1,\infty)$, we define the Sobolev norms,
\[
\|f\|_{W^{k,p}(M)} := \left(\sum_{i=1}^n \|f^i\|_{W^{k,p}(M)}^p\right)^{1/p},
\]
with
\[
\|f^i\|_{W^{k,p}(M)} := \left(\sum_{j=0}^k \int_M |(\nabla^g)^j f^i|^p \,d\vol_g\right)^{1/p},
\]
where $\nabla^g$ denotes the Levi-Civita connection on $TM$ and all associated bundles (that is, $T^*M$ and their tensor products). If $k=0$, then we denote $\|f\|_{W^{0,p}(M)} = \|f\|_{L^p(M)}$. For $p \in [1,\infty)$ and nonnegative integers $k$, we use \cite[Theorem 3.12]{AdamsFournier} (applied to $W^{k,p}(M;\RR^n)$ and noting that $M$ is a closed manifold) and Banach space duality to define
\[
W^{-k,p'}(M;\RR^n) := \left(W^{k,p}(M;\RR^n)\right)^*,
\]
where $p'\in (1,\infty)$ is the dual exponent defined by $1/p+1/p'=1$. Elements of the Banach space dual $(W^{k,p}(M;\RR^n))^*$ may be characterized via \cite[Section 3.10]{AdamsFournier} as distributions in the Schwartz space $\sD'(M;\RR^n)$ \cite[Section 1.57]{AdamsFournier}.

We note that if $(N,h)$ is real analytic, then the isometric embedding, $(N,h) \subset \RR^n$, may also be chosen to be analytic by the analytic isometric embedding theorem due to Nash \cite{Nash_1966}, with a simplified proof due to Greene and Jacobowitz \cite{Greene_Jacobowitz_1971}).

One says that a map $f \in W^{1,2}(M;N)$ is \emph{weakly harmonic} \cite[Definition 1.4.9]{Helein_harmonic_maps} if it is a critical point of the energy function \eqref{eq:Harmonic_map_energy_functional}, that is
\[
\sE'(f) = 0.
\]
A well-known result due to H\'elein \cite[Theorem 4.1.1]{Helein_harmonic_maps} tells us that if $M$ has dimension $d=2$, then $f \in C^\infty(M;N)$; for $d \geq 3$, regularity results are far more limited --- see, for example, \cite[Theorem 4.3.1]{Helein_harmonic_maps} due to Bethuel.

The statement of the forthcoming Theorem \ref{mainthm:Lojasiewicz-Simon_gradient_inequality_energy_functional_Riemannian_manifolds} includes the most delicate dimension for the source Riemannian manifold, $(M,g)$, namely the case where $M$ has dimension $d=2$. Following the landmark articles by Sacks and Uhlenbeck \cite{Sacks_Uhlenbeck_1981, Sacks_Uhlenbeck_1982}, the case where the domain manifold $M$ has dimension two is well-known to be critical.

%COMMENT-PF-10-21-2015 Does $M$ need to be oriented?
\begin{mainthm}[{\L}ojasiewicz--Simon $W^{k-2,p}$ gradient inequality for the energy function for maps between pairs of Riemannian manifolds]
\label{mainthm:Lojasiewicz-Simon_gradient_inequality_energy_functional_Riemannian_manifolds}
Let $d\geq 2$ and $k \geq 1$ be integers and $p\in (1,\infty)$ be such that $kp > d$. Let $(M,g)$ and $(N,h)$ be closed, smooth Riemannian manifolds, with $M$ of dimension $d$. If $(N,h)$ is real analytic (respectively, $C^\infty$) and $f\in W^{k,p}(M;N)$, then the gradient map for the energy function, $\sE:W^{k,p}(M; N)\to\RR$, in \eqref{eq:Harmonic_map_energy_functional},
\[
W^{k,p}(M; N) \ni f \mapsto \sM(f) \in W^{k-2,p}(M; f^*TN) \subset W^{k-2,p}(M; \RR^n),
\]
is a real analytic (respectively, $C^\infty$) map of Banach spaces. If $(N,h)$ is real analytic and $f_\infty \in W^{k,p}(M; N)$ is a weakly harmonic map, then there are positive constants $Z \in (0, \infty)$, and $\sigma \in (0,1]$, and $\theta \in [1/2,1)$, depending on $f_\infty$, $g$, $h$, $k$, $p$, with the following significance. If $f\in W^{k,p}(M;N)$ obeys the $W^{k,p}$ \emph{{\L}ojasiewicz--Simon neighborhood} condition,
\begin{equation}
\label{eq:Lojasiewicz-Simon_gradient_inequality_harmonic_map_neighborhood_Riemannian_manifold}
\|f - f_\infty\|_{W^{k,p}(M)} < \sigma,
\end{equation}
then the harmonic map energy function
\eqref{eq:Harmonic_map_energy_functional} obeys the
\emph{{\L}ojasiewicz--Simon gradient inequality},
\begin{equation}
\label{eq:Lojasiewicz-Simon_gradient_inequality_harmonic_map_energy_functional_Riemannian_manifold}
\|\sM(f)\|_{W^{k-2,p}(M;f^*TN)}
\geq
Z|\sE(f) - \sE(f_\infty)|^\theta.
\end{equation}
Furthermore, if the hypothesis that $(N,h)$ is analytic is replaced by the condition that $\sE$ is Morse--Bott at $f_\infty$, then \eqref{eq:Lojasiewicz-Simon_gradient_inequality_harmonic_map_energy_functional_Riemannian_manifold} holds with the optimal exponent $\theta=1/2$.
\end{mainthm}

\begin{rmk}[On the hypotheses of Theorem \ref{mainthm:Lojasiewicz-Simon_gradient_inequality_energy_functional_Riemannian_manifolds}]
When $k=d$ and $p=1$, then $W^{d,1}(M;\RR) \subset C(M;\RR)$ is a continuous embedding by \cite[Theorem 4.12]{AdamsFournier} and $W^{d,1}(M;\RR)$ is a Banach algebra by \cite[Theorem 4.39]{AdamsFournier}. In particular, $W^{d,1}(M; N)$ is a real analytic Banach manifold by Proposition \ref{prop:Banach_manifold} and the harmonic map energy function, $\sE: W^{d,1}(M; N) \to \RR$, is real analytic by Proposition \ref{prop:analyticharmd}. However, $\sM'(f_\infty):W^{d,1}(M; f_\infty^*TN) \to W^{d-2,1}(M; f_\infty^*TN)$ need not be a Fredholm operator. Indeed, when $d=2$, failure of the Fredholm property for $\sM'(f_\infty):W^{2,1}(M; f_\infty^*TN) \to L^1(M; f_\infty^*TN)$ (unless $L^1(M; f_\infty^*TN)$ is replaced, for example, by a Hardy $H^1$ space) can be inferred from calculations described by H\'elein \cite{Helein_harmonic_maps}.
\end{rmk}

\begin{rmk}[Previous versions of the {\L}ojasiewicz--Simon gradient inequality for the harmonic map energy function]
Topping \cite[Lemma 1]{Topping_1997} proved a {\L}ojasiewicz-type gradient inequality for maps, $f:S^2 \to S^2$, with small energy, with the latter criterion replacing the usual small $C^{2,\alpha}(M;\RR^n)$ norm criterion of Simon for the difference between a map and a critical point \cite[Theorem 3]{Simon_1983}. Simon uses a $C^2(M;\RR^n)$ norm to measure distance between maps, $f:M \to N$,  in \cite[Equation (4.27)]{Simon_1985}. Topping's result is generalized by Liu and Yang in \cite[Lemma 3.3]{Liu_Yang_2010}. Kwon \cite[Theorem 4.2]{KwonThesis} obtains a {\L}ojasiewicz-type gradient inequality for maps, $f:S^2 \to N$, that are $W^{2,p}(S^2;\RR^n)$-close to a harmonic map, with $1 < p \leq 2$.
\end{rmk}

Theorem \ref{mainthm:Lojasiewicz-Simon_gradient_inequality2} leads in turn to the following refinement of Theorem \ref{mainthm:Lojasiewicz-Simon_gradient_inequality_energy_functional_Riemannian_manifolds}.

\begin{maincor}[{\L}ojasiewicz--Simon $L^2$ gradient inequality for the energy function for maps between pairs of Riemannian manifolds]
\label{maincor:Lojasiewicz-Simon_gradient_inequality_energy_functional_Riemannian_manifolds_L2}
Assume the hypotheses of Theorem \ref{mainthm:Lojasiewicz-Simon_gradient_inequality_energy_functional_Riemannian_manifolds} and, in addition, require that $k$ and $p$ obey
\begin{enumerate}
\item \label{item:maincor_LS_grad_ineq_L2_d=2_k=1}
$d=2$ and $k=1$ and $2 < p < \infty$; or

\item \label{item:maincor_LS_grad_ineq_L2_d=3_k=1}
$d=3$ and $k=1$ and $3 < p \leq 6$; or

\item \label{item:maincor_LS_grad_ineq_L2_dgeq2_kgeq2}
$d\geq 2$ and $k\geq 2$ and $2 \leq p < \infty$ with $kp > d$.
%COMMENT: Can we not have $1 < p < \infty$ with $kp > d$?
\end{enumerate}
If $f\in W^{k,p}(M;N)$ obeys the $W^{k,p}$ \emph{{\L}ojasiewicz--Simon neighborhood} condition \eqref{eq:Lojasiewicz-Simon_gradient_inequality_harmonic_map_neighborhood_Riemannian_manifold}, then the harmonic map energy function \eqref{eq:Harmonic_map_energy_functional} obeys the \emph{{\L}ojasiewicz--Simon $L^2$ gradient inequality},
\begin{equation}
\label{eq:Lojasiewicz-Simon_gradient_inequality_harmonic_map_energy_functional_Riemannian_manifold_L2}
\|\sM(f)\|_{L^2(M;f^*TN)}
\geq
Z|\sE(f) - \sE(f_\infty)|^\theta.
\end{equation}
Furthermore, if the hypothesis that $(N,h)$ is analytic is replaced by the condition that $\sE$ is Morse--Bott at $f_\infty$, then \eqref{eq:Lojasiewicz-Simon_gradient_inequality_harmonic_map_energy_functional_Riemannian_manifold_L2} holds with the optimal exponent $\theta=1/2$.
\end{maincor}

\begin{rmk}[Application to proof of Simon's $L^2$ gradient inequality for the energy function for maps between pairs of Riemannian manifolds]
Simon's statement \cite[Theorem 3]{Simon_1983}, \cite[Equation (4.27)]{Simon_1985} of the $L^2$ gradient inequality for the energy function for maps from a closed Riemannian manifold into a closed, real analytic Riemannian manifold is identical to that of Corollary \ref{maincor:Lojasiewicz-Simon_gradient_inequality_energy_functional_Riemannian_manifolds_L2}, except that it applies to $C^{2,\lambda}$ (rather than $W^{k,p}$) maps (for $\lambda \in (0,1)$) and the condition \eqref{eq:Lojasiewicz-Simon_gradient_inequality_harmonic_map_neighborhood_Riemannian_manifold} is replaced by
\[
\|f - f_\infty\|_{C^{2,\lambda}(M;\RR^n)} < \sigma,
\]
Simon's \cite[Theorem 3]{Simon_1983}, \cite[Equation (4.27)]{Simon_1985} follows immediately from Corollary \ref{maincor:Lojasiewicz-Simon_gradient_inequality_energy_functional_Riemannian_manifolds_L2} and the Sobolev Embedding \cite[Theorem 4.12]{AdamsFournier} by choosing $k\geq 1$ and $p \in (1,\infty)$ with $kp>d$ so that there is a continuous Sobolev embedding, $C^{2,\lambda}(M;\RR) \subset W^{k,p}(M;\RR)$ and thus
\[
\|f - f_\infty\|_{W^{k,p}(M;\RR^n)} \leq C\|f - f_\infty\|_{C^{2,\lambda}(M;\RR^n)},
\]
for some constant, $C = C(g,h,k,p,\lambda) \in [1,\infty)$.
\end{rmk}

\begin{rmk}[Exclusion of the case $d\geq 4$ and $k=1$ in Corollary \ref{maincor:Lojasiewicz-Simon_gradient_inequality_energy_functional_Riemannian_manifolds_L2}]
\label{rmk:Lojasiewicz-Simon_gradient_inequality_energy_functional_manifolds_L2_exclude_d_geq_4}
The proofs of Items \eqref{item:maincor_LS_grad_ineq_L2_d=2_k=1} and \eqref{item:maincor_LS_grad_ineq_L2_d=3_k=1} require that $p$ obey $(p')^* = dp/(d(p-1)-p) \geq 2$, namely $dp \geq 2d(p-1) - 2p = 2dp-2d-2p$, or equivalently, $dp \leq 2d+2p$, or equivalently, $p(d-2) \leq 2d$, that is, $p \leq 2d/(d-2)$. But the condition $kp>d$ for $k=1$ implies $p>d$ and so $d$ must obey $d < 2d/(d-2)$, that is $d-2 < 2$ or $d<4$.
\end{rmk}

\begin{rmk}[Relaxing the condition $p \geq 2$ in Item \eqref{item:maincor_LS_grad_ineq_L2_dgeq2_kgeq2} of Corollary \ref{maincor:Lojasiewicz-Simon_gradient_inequality_energy_functional_Riemannian_manifolds_L2}]
When $k \geq 3$, the condition $p\geq 2$ in Item \eqref{item:maincor_LS_grad_ineq_L2_dgeq2_kgeq2} of Corollary \ref{maincor:Lojasiewicz-Simon_gradient_inequality_energy_functional_Riemannian_manifolds_L2} can be relaxed using the Sobolev embedding \cite[Theorem 4.12]{AdamsFournier}.
\end{rmk}

\subsection{Outline of the article}
\label{subsec:Outline}
In Section \ref{sec:Lojasiewicz-Simon_gradient_inequality_harmonic_map_energy_functional}, we establish the {\L}ojasiewicz--Simon gradient inequality for the harmonic map energy function, proving Theorem \ref{mainthm:Lojasiewicz-Simon_gradient_inequality_energy_functional_Riemannian_manifolds}. We refer the reader to \cite[Appendix C]{Feehan_lojasiewicz_inequality_all_dimensions_morse-bott} for a review of the relationship between the Morse--Bott property and the integrability in the setting of harmonic maps. Lastly, \cite[Appendix
D]{Feehan_Maridakis_Lojasiewicz-Simon_coupled_Yang-Mills_v6} includes an explanation of why Theorem \ref{mainthm:Lojasiewicz-Simon_gradient_inequality2} is so useful in applications to questions of global existence and convergence of gradient flows for energy functions on Banach spaces under the validity of the {\L}ojasiewicz--Simon gradient inequality.

We refer the reader to \cite[Appendix A]{Feehan_Maridakis_Lojasiewicz-Simon_harmonic_maps_v5} for a review of the relationship between the Morse--Bott property and the integrability in the setting of harmonic maps. In \cite[Appendix B]{Feehan_Maridakis_Lojasiewicz-Simon_harmonic_maps_v5}, we give a review of Huang's \cite[Theorem 2.4.2 (i)]{Huang_2006} for the {\L}ojasiewcz--Simon gradient inequality for analytic functions on Banach spaces. Next, \cite[Appendix D]{Feehan_Maridakis_Lojasiewicz-Simon_harmonic_maps_v5} provides a few elementary observations from linear functional analysis that illuminate the hypotheses of Theorem \ref{mainthm:Lojasiewicz-Simon_gradient_inequality}.

\subsection{Notation and conventions}
\label{subsec:Notation}
For the notation of function spaces, we follow Adams and Fournier \cite{AdamsFournier}, and for functional analysis, Brezis \cite{Brezis} and Rudin \cite{Rudin}. We let $\NN:=\left\{0,1,2,3,\ldots\right\}$ denote the set of non-negative integers. We use $C=C(*,\ldots,*)$ to denote a constant which depends at most on the quantities appearing on the parentheses. In a given context, a constant denoted by $C$ may have different values depending on the same set of arguments and may increase from one inequality to the next. If $\sX, \sY$ is a pair of Banach spaces, then $\sL(\sX,\sY)$ denotes the Banach space of all continuous linear operators from $\sX$ to $\sY$. We denote the continuous dual space of $\sX$ by $\sX^* = \sL(\sX,\RR)$. We write $\alpha(x) = \langle x, \alpha \rangle_{\sX\times\sX^*}$ for the pairing between $\sX$ and its dual space, where $x \in \sX$ and $\alpha \in \sX^*$. If $T \in \sL(\sX, \sY)$, then its adjoint is denoted by $T^* \in \sL(\sY^*,\sX^*)$, where $(T^*\beta)(x) := \beta(Tx)$ for all $x \in \sX$ and $\beta \in \sY^*$.

\subsection{Acknowledgments}
\label{subsec:Acknowledgments}
Paul Feehan is very grateful to the Max Planck Institute for Mathematics, Bonn, and the Institute for Advanced Study, Princeton, for their support during the preparation of this article. He would like to thank Peter Tak{\'a}{\v{c}} for many helpful conversations regarding the {\L}ojasiewicz--Simon gradient inequality, for explaining his proof of \cite[Proposition 6.1]{Feireisl_Takac_2001} and how it can be generalized as described in this article, and for his kindness when hosting his visit to the Universit{\"a}t R{\"o}stock. He would also like to thank Brendan Owens for several useful conversations and his generosity when hosting his visit to the University of Glasgow. He thanks Haim Brezis for helpful comments on $L\log L$ spaces, Alessandro Carlotto for useful comments regarding the integrability of critical points of the Yamabe function, Sagun Chanillo for detailed and generous assistance with Hardy spaces, and Brendan Owens and Chris Woodward for helpful communications and comments regarding Morse--Bott theory. Both authors are very grateful to one researcher for pointing out an error in an earlier statement and proof of Theorem \ref{mainthm:Lojasiewicz-Simon_gradient_inequality_energy_functional_Riemannian_manifolds}.

\section{{\L}ojasiewicz--Simon gradient inequalities for the harmonic map energy function}
\label{sec:Lojasiewicz-Simon_gradient_inequality_harmonic_map_energy_functional}
Our overall goal in this section is to prove Theorem \ref{mainthm:Lojasiewicz-Simon_gradient_inequality_energy_functional_Riemannian_manifolds}, the {\L}ojasiewicz--Simon gradient inequality for the harmonic map energy function $\sE$ in the cases where $(N,h)$ is a closed, real analytic, Riemannian target manifold or $\sE$ is Morse--Bott at a critical point $f_\infty$, under the hypotheses that $f$ belongs to a traditional $W^{k,p}$ or an $L^2$ {\L}ojasiewicz--Simon neighborhood of $f_\infty$. By way of preparation we prove in Section \ref{subsec:Real_analytic_manifold_structure_Sobolev_space_maps} that $W^{k,p}(M;N)$ is a real analytic (respectively, $C^\infty$) Banach manifold when $(N,h)$ is real analytic (respectively, $C^\infty$). In Section \ref{subsec:Smoothness_analyticity_harmonic_map_energy_functional}, we prove that $\sE$ is real analytic (respectively, $C^\infty$) when $(N,h)$ is real analytic (respectively, $C^\infty$). In Section \ref{subsec:Application_Lojasiewicz-Simon_gradient_inequality_harmonic_map_energy_functional} we complete the proof of Theorem \ref{mainthm:Lojasiewicz-Simon_gradient_inequality_energy_functional_Riemannian_manifolds}, giving the $W^{k-2,p}$ {\L}ojasiewicz--Simon gradient inequality for the harmonic map energy function. Finally, in Section \ref{subsec:Application_Lojasiewicz-Simon_gradient_inequality_harmonic_map_energy_functional_L2}, we prove Corollary \ref{maincor:Lojasiewicz-Simon_gradient_inequality_energy_functional_Riemannian_manifolds_L2}, giving the $L^2$ {\L}ojasiewicz--Simon gradient inequality for the harmonic map energy function.

\subsection{Real analytic manifold structure on the space of Sobolev maps}
\label{subsec:Real_analytic_manifold_structure_Sobolev_space_maps}
The \cite[Theorems 13.5 and 13.6]{PalaisFoundationGlobal} due to Palais imply that the space $W^{k,p}(M;N)$ of $W^{k,p}$ maps (with $kp>d$) from a closed, $C^\infty$ manifold $M$ of dimension $d$ into a closed, $C^\infty$ manifold $N$ can be endowed with the structure of a $C^\infty$ manifold by choosing the fiber bundle, $E \to M$, considered by Palais to be the product $E = M\times N$ and viewing maps $f:M\to N$ as sections of $E\to M$. In particular, \cite[Theorem 13.5]{PalaisFoundationGlobal} establishes the $C^\infty$ structure while \cite[Theorem 13.6]{PalaisFoundationGlobal} identifies the tangent spaces.

While other authors have also considered the smooth manifold structure of spaces of maps between smooth manifolds (see Eichhorn \cite{Eichhorn_1993}, Krikorian \cite{Krikorian_1972}, or Piccione and Tausk \cite{Piccione_Tausk_2001}) or approximation properties (see Bethuel \cite{Bethuel_1991}), none appear to have considered the specific question of interest to us here, namely, the real analytic manifold structure of the space of Sobolev maps from a closed, Riemannian, $C^\infty$ manifold into a closed, real analytic, Riemannian manifold. Moreover, the question does not appear to be considered directly in standard references for harmonic maps (such as H{\'e}lein \cite{Helein_harmonic_maps}, Jost \cite{Jost_riemannian_geometry_geometric_analysis}, or Struwe \cite{Struwe_1996, Struwe_variational_methods}, or references cited therein). Those consideration aside, it will be useful to establish this property directly and, in so doing, develop the framework we shall need to prove the {\L}ojasiewicz--Simon gradient inequality for the harmonic map energy function (Theorem \ref{mainthm:Lojasiewicz-Simon_gradient_inequality_energy_functional_Riemannian_manifolds}).

We shall assume the notation and conventions of Section \ref{subsec:Lojasiewicz-Simon_gradient_inequality_harmonic_map_functional}, so $(M,g)$ is a closed, smooth Riemannian manifold of dimension $d$ and $(N,h)$ is a closed, real analytic (or $C^\infty$), Riemannian, manifold that is embedded analytically (or smoothly) and isometrically in $\RR^n$. We shall view $N$ as a subset of $\RR^n$ with Riemannian metric $h$ given by the restriction of the Euclidean metric. Therefore, a map $f: M \to N$ will be viewed as a map $f:M\to \RR^n$ such that $f(x)\in N$ for every $x\in M$ and similarly a section $Y:N \to TN$ will be viewed as a map $Y :N \to \RR^{2n}$ such that $Y(y) \in T_yN$ for every $y\in N$.

The space of maps,
\[
W^{k,p}(M; N) := \{f\in W^{k,p}(M;\RR^n): f(x)\in N, \text{ for a.e. } x\in M\},
\]
inherits the Sobolev norm from $W^{k,p}(M;\RR^n)$ and by \cite[Theorem 4.12]{AdamsFournier} embeds continuously into the Banach space of continuous maps, $C(M;\RR^n)$, when $kp>d$ or $p=1$ and $k=d$. Furthermore, for this range of exponents, $W^{k,p}(M; N)$ can be given the structure of a real analytic Banach manifold, as we prove in Proposition~\ref{prop:Banach_manifold}. A definition of coordinate charts on $W^{k,p}(M; N)$ is given \cite[Section~4.3]{KwonThesis}, which we now recall.

Let $\sO$ denote a normal tubular neighborhood \cite[p. 110]{Hirsch} of radius $\delta_0$ of $N$ in $\RR^n$, so $\delta_0 \in (0,1]$ is sufficiently small that there is a well-defined projection map, $\pi_h : \sO \to N \subset \RR^n$, from $\sO$ to the nearest point of $N$. When $y\in N$, the value $\pi_h(y + \eta)$ is well defined for $\eta\in \RR^n$ with $|\eta|<\delta_0$ and the differential,
\begin{equation}
\label{eq:Differential_nearest_point_projection_as_orthogonal_projection}
d\pi_h(y+\eta): T_{y+\eta}\RR^n\cong \RR^n \to T_{\pi_h(y+\eta)}N,
\end{equation}
is given by orthogonal projection (see Simon \cite[Section 2.12.3, Theorem 1]{Simon_1996}).

\begin{lem}[Analytic diffeomorphism of a neighborhood of the zero-section of the tangent bundle onto an open neighborhood of the diagonal]
\label{lem:tubular_neighborhood}
Let $(N,h)$ be a closed, real analytic, Riemannian manifold that is analytically and isometrically embedded in $\RR^n$ and let $(\pi_h,\sO)$ be a normal tubular neighborhood of radius $\delta_0$ of $N \subset \RR^n$, where $\pi_h : \sO \to N \subset \RR^n$ is the projection to the nearest point of $N$.  Then there is a constant $\delta_1 \in (0, \delta_0]$ such that the map,
\begin{equation}
\label{eq:fact}
\Phi: \left\{(y, \eta)\in TN: |\eta|< \delta_1\right\}
\to N\times N\subset \RR^{2n},
\quad (y,\eta) \mapsto (y,\pi_h(y+\eta)),
\end{equation}
is an analytic diffeomorphism onto an open neighborhood of the diagonal of $N\times N \subset \RR^{2n}$.
\end{lem}

\begin{proof}
For each $y \in N$, we have $\Phi(y,0) = (y,y) \in \diag(N\times N)$, where $\diag(N\times N)$ denotes the diagonal of $N \times N$. Moreover, $T_{(y,0)}(TN) = T_yN \times T_yN$ and the differential $d\Phi(y,0):(TN)_{(y,0)} \to T_yN \times T_yN$ is given by $(\zeta,\eta) \mapsto (\zeta, d\pi_h(y)(\eta)) = (\zeta,\eta)$, that is, the identity. By \cite[Theorem 5.1 and remark following proof, p. 110]{Hirsch}, the projection $\pi_h$ is $C^\infty$ and by replacing the role of the $C^\infty$ Inverse Function Theorem in its proof by the real analytic counterpart, one can show that $\pi_h$ is real analytic; see \cite[Section 2.12.3, Theorem 1]{Simon_1996} due to Simon for a proof. Thus $\Phi$ is real analytic and the Analytic Inverse Function Theorem now yields the conclusion of the lemma.
\end{proof}

For a map $f\in W^{k,p}(M; N)$, we note that, because there is a continuous Sobolev embedding, $W^{k,p}(M;\RR^n)\subset C(M;\RR^n)$, for $kp>d$ by \cite[Theorem 4.12]{AdamsFournier}, we can regard $f$ as a continuous map $f : M\to \RR^n$ such that $f(M)\subset N$. Let $B_f(\delta)$ denote the ball of center zero and radius $\delta>0$ in the Sobolev space, $W^{k,p}(M;f^*TN)$, and denote
\begin{equation}
\label{eq:Ball_in_Wkp_Vf}
\sU_f := B_f(\kappa(f)^{-1}\delta) \subset W^{k,p}(M;f^*TN),
\end{equation}
where $\kappa(f)$ is the norm of the Sobolev embedding, $W^{k,p}(M;f^*TN)\subset C(M;f^*TN)$.

\begin{prop}[Banach manifold structure on the Sobolev space of maps between Riemannian manifolds]
\label{prop:Banach_manifold}
Let $d\geq 1$ and $k\geq 1$ be integers and $p\in [1,\infty)$ be such that
\[
kp>d \quad \text{or}\quad k=d \text{ and } p=1.
\]
Let $(M,g)$ be a closed, Riemannian, $C^\infty$ manifold of dimension $d$ and $(N,h)$ be a closed, real analytic, Riemannian, manifold that is isometrically and analytically embedded in $\RR^n$ and identified with its image. Then the space of maps, $W^{k,p}(M; N)$, has the structure of a real analytic Banach manifold and for each $f\in W^{k,p}(M; N)$, there is a constant $\delta = \delta(N,h) \in (0,1]$
%COMMENT - For arXiv v2, why is \delta independent of g,f,k,p?
such that, with the definition of $\sU_f$ from \eqref{eq:Ball_in_Wkp_Vf}, the map,
\begin{equation}
\label{eq:Phi_f}
\Phi_f: W^{k,p}(M;f^*TN) \supset \sU_f \to W^{k,p}(M; N) , \quad u \mapsto \pi_h(f + u),
\end{equation}
defines an inverse coordinate chart on an open neighborhood of $f \in W^{k,p}(M; N)$ and a real analytic manifold structure on $W^{k,p}(M; N)$. Finally, if the hypothesis that $(N,h)$ is real analytic is relaxed to the hypothesis that it is $C^\infty$, then $W^{k,p}(M; N)$ inherits the structure of a $C^\infty$ manifold.
\end{prop}

\begin{proof}
Because $N \subset \RR^n$ is a real analytic submanifold, it follows from arguments of Palais \cite[Chapter 13]{PalaisFoundationGlobal} that $W^{k,p}(M;N)$ is a real analytic submanifold of the Banach space $W^{k,p}(M;\RR^{2n})$. Because Palais treats the $C^\infty$ but not explicitly the real analytic case, we provide details.

Let $f \in W^{k,p}(M;N)$ and define an open ball with center $f$ and radius $\eps \in (0,1]$,
\[
\BB_f(\eps) := \{v \in W^{k,p}(M;\RR^{2n}): \|v-f\|_{W^{k,p}(M)} < \eps\},
\]
Recall from Lemma \ref{lem:tubular_neighborhood}, that the assignment $\Phi(y,\eta) = (y,\pi_h(y+\eta))$ defines an analytic diffeomorphism from an open neighborhood of the zero section $N \subset TN$ onto an open neighborhood of the diagonal $N \subset N\times N \subset \RR^{2n}$. In particular, the assignment $\Phi_f(u) = \pi_h(f+u)$, for $u$ belonging to a small enough open ball, $B_f(\delta_2)$, centered at the origin in $W^{k,p}(M;f^*TN)$, defines a real analytic embedding of $B_f(\delta_2)$ into $W^{k,p}(M;\RR^{2n})$ and onto a relatively open subset, $\Phi_f(B_f(\delta_2)) \subset W^{k,p}(M;N)$. Thus, for small enough $\eps$,
\[
\BB_f(\eps) \cap W^{k,p}(M;N) \subset \Phi_f(B_f(\delta_2)).
\]
The assignment $\Phi_f(u) = \pi_h(f+u) \in  W^{k,p}(M;N)$, for $u \in B_f(\delta_2)$, may be regarded as the restriction of the real analytic map,
\[
W^{k,p}(M;\RR^{2n}) \ni u \mapsto \pi_h(f + u) \in W^{k,p}(M;\RR^{2n}).
\]
Therefore, the collection of inverse maps, defined by each $f \in W^{k,p}(M;N)$,
\[
\Phi_f^{-1}: \BB_f(\eps) \cap W^{k,p}(M;N) \to W^{k,p}(M;f^*TN),
\]
defines an atlas for a real analytic manifold structure on $W^{k,p}(M;N)$ as a real analytic submanifold of $W^{k,p}(M;\RR^{2n})$.
%COMMENT-PF-10-12-2015 May need to standardize ball notation to B_center(radius) everywhere.

Lastly, we relax the assumption of real analyticity and require only that $(N,h)$ be a $C^\infty$ closed, Riemannian manifold and isometrically and smoothly embedded in $\RR^n$ and identified with its image. The conclusion that $W^{k,p}(M;N)$ is a $C^\infty$ manifold is immediate from the proof in the real analytic case by just replacing real analytic with $C^\infty$ diffeomorphisms.
\end{proof}

\begin{rmk}[Identification of the tangent spaces]
\label{rmk:tangent}
The existence of a $C^\infty$ Banach manifold structure for $W^{k,p}(M;N)$ in the case of a smooth isometric embedding $(N,h)\subset \RR^n$ is also provided in \cite[Theorem 13.5]{PalaisFoundationGlobal}. In \cite[Theorem 13.6]{PalaisFoundationGlobal} the Banach space $W^{k,p}(M;f^*TN)$ is identified as the tangent space of the Banach manifold $W^{k,p}(M;N)$ at the point $f$. Note that for $f\in W^{k,p}(M;N)$, the differential $(d\Phi_f)(0) : W^{k,p}(M;f^*TN)\to T_fW^{k,p}(M;N)$ is the identity map.
\end{rmk}

\begin{rmk}[Properties of coordinate charts]
\label{rmk:chart_properties}
For the inverse coordinate chart $(\Phi_f, \sU_f)$ and $u\in \sU_f$ with $f_1 := \pi_h(f+u) \in W^{k,p}(M; N)$, the differential,
\[
(d\Phi_f)(u) : W^{k,p}(M;f^*TN) \to W^{k,p}(M;f_1^*TN) \subset W^{k,p}(M;\RR^n),
\]
is an isomorphism of Banach spaces. By choosing $\delta \in (0,1]$ in Proposition \ref{prop:Banach_manifold} sufficiently small we find that the norm of the operator,
\[
(d\Phi_f)(u) - (d\Phi_f)(0) : W^{k,p}(M;f^*TN) \to W^{k,p}(M;\RR^n),
\]
obeys
\[
\|(d\Phi_f)(u) - (d\Phi_f)(0)\| \leq 1,\quad \forall\, u\in \sU_f,
\]
and therefore $C_3 := \sup_{u \in \sU_f} \|(d\Phi_f)(u)\| \leq 2$. By applying the Mean Value Theorem to $\Phi_f$ and its inverse, we obtain
\begin{equation}
\label{eq:Kwon_4-7}
C_4^{-1}\|f-f_1\|_{W^{k,p}(M)} \leq  \|u\|_{W^{k,p}(M)} \leq C_4  \|f-f_1\|_{W^{k,p}(M)}
\end{equation}
for every $f\in W^{k,p}(M;N)$ and every  $u\in W^{k,p}(M;f^*TN)$ with $f_1 = \pi_h(f+u)$, where $C_4 \geq C_3$ depends on $(N,h)$ and $f$. (Compare \cite[Inequality (4.7)]{KwonThesis}.)
\end{rmk}

\subsection{Smoothness and analyticity of the harmonic map energy function}
\label{subsec:Smoothness_analyticity_harmonic_map_energy_functional}
We shall assume the notation and conventions of Section \ref{subsec:Real_analytic_manifold_structure_Sobolev_space_maps}. Recall Definition~\ref{defn:Harmonic_map_energy_functional} of the harmonic map energy function,
\[
\sE : W^{k,p}(M; N) \to \RR, \quad f \mapsto \frac{1}{2}\int_M |df|^2\, d\vol_g,
\]
and consider the pullback of $\sE$ by a local coordinate chart on $ W^{k,p}(M; N)$,
\begin{equation}
\label{eq:Pull_back_harmonic_map_energy_functional}
\sE\circ \Phi_f : W^{k,p}(M;f^*TN) \supset \sU_f \ni u \mapsto \frac{1}{2}\int_M| d(\pi_h(f + u))|^2\, d\vol_g \in \RR.
\end{equation}
We now establish the following proposition.

\begin{prop}[Smoothness and analyticity of the harmonic map energy function]
\label{prop:analyticharmd}
Let $d\geq 2$ and $k \geq 1$ be integers and $p\in [1,\infty)$ be such that
\[
kp > d \quad\text{or}\quad k=d \text{ and } p=1.
\]
Let $(M,g)$ and $(N,h)$ be closed, smooth Riemannian manifolds with $(N,h)$ real analytic and analytically and isometrically embedded in $\RR^n$ and identified with its image. If $f\in W^{k,p}(M;N)$, then $\sE\circ\Phi_f : W^{k,p}(M;f^*TN) \supset \sU_f \to \RR$ in \eqref{eq:Pull_back_harmonic_map_energy_functional} is a real analytic map, where $\sU_f \subset W^{k,p}(M;f^*TN)$ is as in \eqref{eq:Ball_in_Wkp_Vf} and the image of a coordinate neighborhood in $W^{k,p}(M;N)$. In particular, the function,
\[
\sE: W^{k,p}(M;N) \to \RR,
\]
is real analytic. Finally, if the hypothesis that $(N,h)$ is real analytic is relaxed to the hypothesis that it is $C^\infty$, then the function $\sE: W^{k,p}(M;N) \to \RR$ is $C^\infty$.
\end{prop}

\begin{proof}
Our hypotheses on $d,k,p$ ensure that there is a continuous Sobolev embedding, $W^{k,p}(M;N) \subset C(M;N)$ by \cite[Theorem 4.12]{AdamsFournier} and that $W^{k,p}(M;\RR)$ is a Banach algebra by \cite[Theorem 4.39]{AdamsFournier}. By hypothesis, $f \in W^{k,p}(M;N)$, so $f \in C(M;N)$. We view $N \subset \RR^n$ as isometrically and real analytically embedded as the zero section of its tangent bundle, $TN$, and which is in turn isometrically and real analytically embedded in $\RR^{2n}$ and identified with its image. Moreover, if $u \in W^{k,p}(M;f^*TN)$, then $u \in C(M;f^*TN)$.

As in Lemma \ref{lem:tubular_neighborhood}, let $(\pi_h,\sO)$ be a normal tubular neighborhood of $N \subset \RR^n$ of radius $\delta_0 \in (0,1]$. Because the nearest-point projection map, $\pi_h:\sO\subset\RR^n \to N$, is real analytic, its differential, $(d\pi_h)(y) \in \Hom_\RR(\RR^n,T_yN) \subset \End_\RR(\RR^n)$, is a real analytic function of $y \in \sO$ and $d\pi_h(y):\RR^n\to\RR^n$ is orthogonal projection. We choose $\eps \in (0,1]$ small enough that $d\pi_h(y+z)$ has a power series expansion centered at each point $y \in \sO$ with radius of convergence $\eps$,
\[
d\pi_h(y+z) = \sum_{m=0}^\infty a_m(y) z^m, \quad\forall\, y,z \in \RR^n \text{ with } |z| < \eps,
\]
where (see, for example, Whittlesey \cite{Whittlesey_1965} in the case of analytic maps of Banach spaces), for each $y \in \sO$, the coefficients $a_m(y;z_1,\ldots,z_m)$ are continuous, multilinear, symmetric maps of $(\RR^n)^m$ into $\End_\RR(\RR^n)$ and we abbreviate $a_m(y;z,\ldots,z) = a_m(y) z^m$. The coefficient maps, $a_m(y)$, are (analytic) functions of $y \in \sO$, intrinsically defined as derivatives of $d\pi_h$ at $y \in \sO$. We shall use the convergent power series for $d\pi_h(y+z)$, in terms of $z$ with $|z| < \eps$, to determine a convergent power series for $(\sE\circ\Phi_f)(u)$ in \eqref{eq:Pull_back_harmonic_map_energy_functional}, namely
\[
(\sE\circ\Phi_f)(u) = \frac{1}{2}\int_M |d(\pi_h(f + u))|^2\,d\vol_g = \frac{1}{2}\int_M |d\pi_h(f+u)(df + du)|^2\,d\vol_g,
\]
in terms of $u \in W^{k,p}(M;f^*TN)$ with $\|u\|_{W^{k,p}(M)} < \delta$, where $\delta = \eps/\kappa$ and $\kappa = \kappa(f,g,h)$ is the norm of the Sobolev embedding, $W^{k,p}(M;f^*TN) \subset C(M;f^*TN)$. Recall that
\[
d\pi_h(f+u)(df + du)|_x = d\pi_h(f(x)+u(x))(df(x) + du(x)), \quad\forall\, x\in M,
\]
where $f(x)+u(x) \in \sO$ and $f(x) + du(x) \in T_{f(x)}N$. We have the pointwise identity,
\[
|d\pi_h(f + u)(df+du)|^2
=
\left|\left(\sum_{m=0}^\infty a_m(f) u^m\right)(df+du)\right|^2  \quad\text{on } M,
\]
and thus,
\begin{multline*}
|d(\pi_h(f + u))|^2
=
\sum_{m=0}^\infty |(a_m(f) u^m)(df+du)|^2
\\
+ 2\sum_{l=1}^\infty \sum_{m=0}^\infty \langle |(a_m(f) u^m)(df+du), (a_{m+l}(f) u^{m+l})(df+du)\rangle \quad\text{on } M.
\end{multline*}
After substituting the preceding expression and noting that $M$ is compact and that all integrands are continuous functions on $M$, the Lebesgue Dominated Convergence Theorem yields a convergent power series as a function of $u \in W^{k,p}(M;f^*TN)$ with $\|u\|_{W^{k,p}(M)} < \delta$,
\[
(\sE\circ\Phi_f)(u)
=
\frac{1}{2}\int_M \left|\left(\sum_{m=0}^\infty a_m(f) u^m\right)(df+du)\right|^2 \,d\vol_g,
\]
and thus $(\sE\circ\Phi_f)(u)$ is an analytic function of $u \in W^{k,p}(M;f^*TN)$ with $\|u\|_{W^{k,p}(M)} < \delta$.

We now relax the assumption of real analyticity of $(N,h)$ and require only that $(N,h)$ be a $C^\infty$ closed, Riemannian manifold and isometrically and smoothly embedded in $\RR^n$ and identified with its image. The conclusion that the map $\sE\circ\Phi_f:W^{k,p}(M;f^*TN) \supset \sU_f \to \RR$ is $C^\infty$ is immediate from the fact that $W^{k,p}(M;f^*TN) \subset C(M;f^*TN)$ because the pointwise expressions for $|d\pi_h(f(x)+u(x))(df(x) + du(x))|^2$, for $x\in M$, and all higher-order derivatives with respect to $z = u(x) \in \sO \subset \RR^n$ will be continuous functions on the compact manifold, $M$.
\end{proof}

\subsection{Application to the $W^{k-2,p}$ {\L}ojasiewicz--Simon gradient inequality for the harmonic map energy function}
\label{subsec:Application_Lojasiewicz-Simon_gradient_inequality_harmonic_map_energy_functional}
%COMMENT For arXiv v2, check whether use of \|\cdot\|_{W^{k,p}(M;\RR^n)} versus \|\cdot\|_{W^{k,p}(M;V_f)} is consistent
We continue to assume the notation and conventions of Section \ref{subsec:Real_analytic_manifold_structure_Sobolev_space_maps}. The covariant derivative, with respect to the Levi-Civita connection for the Riemannian metric $h$ on $N$, of a vector field $Y \in C^\infty(TN)$ is given by
\begin{equation}
\label{eq:convariant_derivative_TN}
(\nabla^hY)_y = d\pi_h(y)(dY),
\end{equation}
where $\pi_h$ is as discussed around \eqref{eq:Differential_nearest_point_projection_as_orthogonal_projection}
and the second fundamental form \cite[Definition 4.7.2]{Jost_riemannian_geometry_geometric_analysis} of the embedding $N\subset \RR^n$ is given by
\begin{equation}
\label{eq:2ndfundamentalform}
A_h(X,Y) := \left(\nabla_X^h Y\right)^\perp = dY(X) - d\pi_h(dY(X)),
\quad\forall\, X, Y \in C^\infty(TN),
\end{equation}
where $dY$ is the differential of the map $Y: N\to \RR^{2n}$ and we recall from \eqref{eq:Differential_nearest_point_projection_as_orthogonal_projection} that $d\pi_h(y):\RR^n \to T_yN$ is orthogonal projection, so $d\pi_h(y)^\perp = \id-d\pi_h(y):\RR^n \to (T_yN)^\perp$ is orthogonal projection onto the normal plane. By \cite[Lemma 4.7.2]{Jost_riemannian_geometry_geometric_analysis} we know that $A_h(y):T_yN\times T_yN \to (T_yN)^\perp$ is a symmetric bilinear form with values in the normal space, for all $y\in N$.

The forthcoming Lemma \ref{lem:Euler-Lagrange_equation_and_gradient_harmonic_map_energy_functional} is of course well-known (for example, see \cite[Lemma 1.4.10]{Helein_harmonic_maps}), but it will be useful to recall the proof since conventions vary in the literature and we shall subsequently require the ingredients involved in its proof. (Note that our sign convention for the Laplace operator \eqref{eq:Laplace-Beltrami_operator} is opposite to that of H\'elein in \cite[Equation (1.1)]{Helein_harmonic_maps}.)

\begin{lem}[Euler-Lagrange equation for a harmonic map and gradient of the harmonic map energy functional]
\label{lem:Euler-Lagrange_equation_and_gradient_harmonic_map_energy_functional}
Let $d\geq 1$ and $k\geq 1$ be integers and $p\in [1,\infty)$ be such that
\[
kp>d \quad \text{or}\quad k=d \text{ and } p=1.
\]
Let $(M,g)$ and $(N,h)$ be closed, smooth Riemannian manifolds, with $M$ of dimension $d$ and $(N,h) \subset \RR^n$ a $C^\infty$ isometric embedding. If $f_\infty \in W^{k,p}(M;N)$ is a critical point of the harmonic map energy functional, that is, $\sE'(f_\infty) = 0$, then $f_\infty$ is a weak solution to the Euler-Lagrange equation,
\[
\Delta_g f_\infty - A_h(f_\infty)(df_\infty,df_\infty) = 0 \quad\text{on } M,
\]
where $A_h$ is the second fundamental form defined by the embedding, $(N,h) \subset \RR^n$; moreover,
\begin{align*}
\sM(f) &= d\pi_h(f)\Delta_g f
\\
&= \Delta_g f - A_h(f)(df,df) \in W^{k-2,p}(M;f^*TN),
\end{align*}
is the gradient of $\sE$ at $f \in W^{k,p}(M;N)$ with respect to the inner product on $L^2(M;f^*TN)$,
\[
\sE'(f)(u) = (u,\sM(f))_{L^2(M;f^*TN)}, \quad\forall\, u \in W^{k,p}(M;f^*TN).
\]
\end{lem}

\begin{proof}
We consider variations of $f \in W^{k,p}(M;N)$ of the form $f_t = \pi_h(f+tu)$, for $u \in W^{k,p}(M;f^*TN)$ and $du \in W^{k-1,p}(M;T^*M\otimes \RR^n)$, recall from Section \ref{subsec:Real_analytic_manifold_structure_Sobolev_space_maps} that $d\pi_h(y):\RR^n\to T_yN$ is orthogonal projection for each $y\in N$, and use the definition \eqref{eq:Gradient_harmonic_map_operator} of $\sM$ to compute
\begin{align*}
\left(u, \sM(f)\right)_{L^2(M)}
&= \sE'(f)(u) = \left.\frac{d}{dt}\sE(\pi_h(f+tu))\right|_{t=0}
\\
&= \frac{1}{2}\left.\frac{d}{dt}(d\pi_h(f+tu), d\pi_h(f+tu))_{L^2(M)} \right|_{t=0}
\\
&= (d\pi_h(f)(du), df)_{L^2(M)} = (du, df)_{L^2(M)},
\end{align*}
where $d\pi_h(f):\underline{\RR}^n \to f^*TN$ is orthogonal projection from the product bundle, $\underline{\RR}^n = M\times\RR^n$, onto the pullback by $f$ of the tangent bundle, $TN$. Thus, writing $\Delta_g f = d^{*,g}df$ for the scalar Laplacian on the components of $f = (f^1,\ldots, f^n): M \to N \subset \RR^n$ implied by the isometric embedding, $(N,h) \subset \RR^n$, we obtain
\begin{align*}
\left(u, \sM(f)\right)_{L^2(M,g)}
&=
(d(d\pi_h(f)(u)), df)_{L^2(M)}
\\
&= (d\pi_h(f)(u), d^{*,g}df)_{L^2(M)}
\\
&= (u, d\pi_h(f)(\Delta_g f) )_{L^2(M)} = (u,\Delta_g f)_{L^2(M)}.
\end{align*}
From \cite[Lemma 1.2.4]{Helein_harmonic_maps} (and noting that our sign convention for the Laplace operator is opposite to that of H\'elein in \cite[Equation (1.1)]{Helein_harmonic_maps}), we have
\begin{equation}
\label{eq:Harmonic_map_pre-Euler-Lagrange_equation}
d\pi_h(f)^\perp(\Delta_g f) = A_h(f)(df,df),
\end{equation}
where $d\pi_h(f)^\perp:\underline{\RR}^n \to f^*T^\perp N$ is orthogonal projection from the product bundle, $\underline{\RR}^n = M\times\RR^n$ onto the pullback by $f$ of the normal bundle, $T^\perp N$, defined by the embedding, $N \subset \RR^n$.

If $f_\infty$ is weakly harmonic, that is, a critical point of $\sE$ and $\sE'(f_\infty)(u) = 0 = (u, \Delta_g f_\infty)_{L^2(M)}$ for all $u \in W^{k,p}(M;f_\infty^*TN)$, then $(\Delta_g f_\infty)(x) \perp T_{f_\infty(x)}N$ for all $x \in M$ (and as in \cite[Lemma 1.4.10 (i)]{Helein_harmonic_maps}). Hence, $d\pi_h(f_\infty)(\Delta_g f_\infty) = 0$ and \eqref{eq:Harmonic_map_pre-Euler-Lagrange_equation} becomes, after replacing $f$ by $f_\infty$,
\[
\Delta_g f_\infty = A_h(f_\infty)(df_\infty,df_\infty),
\]
as claimed (and as in \cite[Lemma 1.4.10 (ii)]{Helein_harmonic_maps}, noting our opposite sign convention for $\Delta_g$). Also,
\begin{align*}
\left(u, \sM(f)\right)_{L^2(M,g)}
&=
(u, d\pi_h(f)(\Delta_g f) )_{L^2(M)}
\\
&= (u,\Delta_g f - d\pi_h(f)^\perp(\Delta_g f) )_{L^2(M)}
\\
&= (u,\Delta_g f - A_h(f)(df,df) )_{L^2(M)} \quad\text{(by \eqref{eq:Harmonic_map_pre-Euler-Lagrange_equation})},
\quad\forall u \in W^{k,p}(M;f^*TN),
\end{align*}
and thus,
\[
\sM(f) = \Delta_g f - A_h(f)(df,df) \in W^{k-2,p}(M;f^*TN),
\]
the gradient of $\sE$ at $f$ with respect to the $L^2$-metric, as claimed.
\end{proof}

Next, we prove\footnote{Although we use the expression $\sM(f) = \Delta_g f - A_h(f)(df,df)$ in this proof of Lemma \ref{lem:Gradient_harmonic_map_operator_smooth}, one could alternatively use the equivalent expression $\sM(f) = d\pi_h(f)\Delta_g f$ and apply the method of proof of Lemma \ref{lem:Isomorphism_Sobolev_spaces_sections_two_nearby_maps}.} a partial analogue for the gradient map, $\sM$, of Proposition \ref{prop:analyticharmd} for the harmonic map energy functional.

\begin{lem}[Smoothness of the gradient map]
\label{lem:Gradient_harmonic_map_operator_smooth}
Let $d\geq 2$ and $k\geq 1$ be integers and $p\in [1,\infty)$ a constant such that
\[
kp>d \quad \text{or}\quad k=d \text{ and } p=1.
\]
Let $(M,g)$ and $(N,h)$ be closed, smooth Riemannian manifolds, with $M$ of dimension $d$. Then the gradient map \eqref{eq:Gradient_harmonic_map_operator} is $C^\infty$,
\[
W^{k,p}(M;N) \ni f \mapsto \sM(f)
\in
W^{k-2,p}(M;\RR^n),
\]
where $\sM(f) \in W^{k-2,p}(M;f^*TN)$.
\end{lem}

\begin{proof}
Recall from Proposition \ref{prop:Banach_manifold} that $W^{k,p}(M;N)$ is a $C^\infty$ Banach manifold and by \eqref{eq:Gradient_harmonic_map_operator},
\[
\sM(f) = d\pi_h(f)\Delta_gf.
\]
We recall from Section \ref{subsec:Real_analytic_manifold_structure_Sobolev_space_maps} that the nearest point projection, $\pi_h:\RR^n \supset \sO \to N$, is a $C^\infty$ map on a normal tubular neighborhood of $N \subset \RR^n$ and that $d\pi_h:
\sO\times\RR^n \to TN$ is $C^\infty$ orthogonal projection. In particular, $d\pi_h \in C^\infty(N; \End(\RR^n))$, while $f \in W^{k,p}(M;N)$ and thus $d\pi_h(f) \in W^{k,p}(M; \End(\RR^n))$ by \cite[Corollary 9.10]{PalaisFoundationGlobal}.

Define $W^{k,p}(M;\sO) :=\{f\in W^{k,p}(M;\RR^n): f(M)\subset \sO\}$, an open subset of the Sobolev space $W^{k,p}(M;\RR^n)$. Recall from \cite[Corollary 9.10]{PalaisFoundationGlobal} that if $S \in C^\infty(\sO;\RR^l)$ for an integer $l\geq 1$, then the map, $W^{k,p}(M;\sO) \ni f \mapsto S(f) \in W^{k,p}(M;\RR^l)$ is continuous under our hypotheses on $d,k,p$ (the case $k=d$ and $p=1$ follows by Palais' argument, though he only considers the case $kp>d$). The Chain Rule for $C^\infty$ maps of Banach spaces implies that the map, $W^{k,p}(M;\sO) \ni f \mapsto S(f) \in W^{k,p}(M;\RR^l)$, is $C^\infty$. In particular,
\begin{equation}
\label{eq:Sobolev_smooth_projection}
W^{k,p}(M;\sO) \supset W^{k,p}(M;N) \ni f \mapsto d\pi(f) \in W^{k,p}(M;\End(\RR^n)),
\end{equation}
is a $C^\infty$ map.

The linear operator, $W^{k,p}(M;\RR^n) \ni v \mapsto \Delta_gv \in W^{k-2,p}(M;\RR^n)$, is bounded and restricts to a $C^\infty$ map, $W^{k,p}(M;N) \ni f \mapsto \Delta_gf \in W^{k-2,p}(M;\RR^n)$.

The Sobolev space, $W^{k,p}(M;\RR)$, is a Banach algebra by \cite[Theorem 4.39]{AdamsFournier}, and the Sobolev multiplication map, $W^{k,p}(M;\RR) \times W^{k-2,p}(M;\RR) \to W^{k-2,p}(M;\RR)$ is bounded by \cite[Corollary 9.7]{PalaisFoundationGlobal} for $k\geq 2$ and the proof of Lemma \ref{lem:Isomorphism_Sobolev_spaces_sections_two_nearby_maps} for $k=1$. The projection, $d\pi_h(f) \in W^{k,p}(M; \Hom(M\times\RR^n, f^*TN))$, acts on $v \in W^{k-2,p}(M;f^*TN)$ by pointwise inner product with coefficients in $W^{k,p}(M;\RR)$ and therefore
\[
W^{k-2,p}(M;\RR^n) \ni v \mapsto d\pi(f)v \in W^{k-2,p}(M;f^*TN) \subset W^{k-2,p}(M;\RR^n),
\]
is a $C^\infty$ map. By combining the preceding observations with the Chain Rule for $C^\infty$ maps of Banach manifolds, we see that
\[
W^{k,p}(M;N)\ni f \mapsto d\pi_h(f)\Delta_gf \in W^{k-2,p}(M;\RR^n)
\]
is a $C^\infty$ map, as claimed. This completes the proof of Lemma \ref{lem:Gradient_harmonic_map_operator_smooth}.
\end{proof}

Before establishing real analyticity of the gradient map, we prove the following elementary

\begin{lem}[Analyticity of maps of Banach spaces]
\label{lem:subanalytic}
Let $\sX, \tilde\sX, \sY$ be Banach spaces, $\sU\subset \sX$ an open subset, $\tilde\sX \subset \sY$ a continuous embedding, and $\sF: \sU \to \tilde\sX$ a $C^\infty$ map. If the composition, $\sF: \sU\to\sY$, is real analytic at $x\in\sU$, then $\sF: \sU \to \tilde\sX$ is also real analytic at $x$.
\end{lem}

\begin{proof}
Recall the definition and notation in Feehan and Maridakis \cite[Section 2.1.2]{Feehan_Maridakis_Lojasiewicz-Simon_Banach} for analytic maps of Banach spaces. Because the composition, $\sF: \sU\to\sY$, is real analytic at $x$, there is a constant $\delta = \delta(x) > 0$ such that the Taylor series,
\[
\sF(x+h) - \sF(x) = \sum_{n=1}^\infty L_n(x)h^n, \quad\forall\, h \in \sX \text{ such that } \|h\|_\sX <\delta,
\]
converges in $\sY$, where, for each $n\in \NN$, we have that $L_n(x): \sX^n\to \sY$ is a bounded linear map, we denote $\sX^n = \sX\times \cdots \times \sX$ ($n$-fold product), and $h^n = (h,\ldots,h) \in \sX^n$.

Since $\sF: \sU \to \tilde\sX$ is $C^\infty$ at $x\in\sU$, then $D^n\sF(x) \in \sL(\sX^n,\tilde\sX)$ for all $n \in \NN$ and the coefficients, $L_n(x) = D^n\sF(x)$, of the Taylor series for $\sF$ centered at $x$ obey
\begin{align*}
\|L_n(x)\|_{\sL(\sX^n,\tilde\sX)}
&=
\sup_{\begin{subarray}{c}\|h_i\|_\sX=1, \\ 1 \leq i \leq n \end{subarray}}
\| L_n(x)(h_1,\ldots, h_n)\|_{\tilde\sX}
\\
&\leq \sup_{\begin{subarray}{c}\|h_i\|_\sX=1, \\ 1 \leq i \leq n \end{subarray}}
C\| L_n(x)(h_1,\ldots, h_n)\|_\sY
\\
&= C \|L_n(x)\|_{\sL(\sX^n,\sY)}, \quad \forall\, n \in \NN,
\end{align*}
where $C \in [1,\infty)$ is the norm of the embedding, $\tilde\sX \subset \sY$.

By definition of analyticity of the composition, $\sF:\sU \to \sY$, there is a constant, $r = r(x) \in (0,\delta]$ such that $\sum_{n=1}^\infty \|L_n(x)\|_{\sL(\sX^n,\sY)} r^n <\infty$. Hence, setting $r_1 = r/C$,
\[
\sum_{n=1}^\infty \|L_n(x)\|_{\sL(\sX^n,\tilde\sX)} r_1^n
\leq
\sum_{n=1}^\infty \|L_n(x)\|_{\sL(\sX^n,\sY)} r^n
< \infty.
\]
Therefore, setting $\delta_1 = \delta/C$, the Taylor series
\[
\sF(x+h) - \sF(x) = \sum_{n=1}^\infty L_n(x) h^n,
\quad \forall\, h \in \sX \text{ such that } \|h\|_\sX <\delta_1,
\]
converges in $\tilde\sX$ and so $\sF: \sU \to \tilde\sX$ is analytic at $x$.
\end{proof}

The converse to Lemma \ref{lem:subanalytic} is an immediate consequence of the analyticity of compositions of analytic maps of Banach spaces \cite[Theorem, p. 1079]{Whittlesey_1965}: If $\sF:\sU\to\tilde\sX$ is real analytic at $x$ and $\tilde\sX \subset \sY$ is a continuous embedding, then the composition, $\sF:\sU\to\sY$, is real analytic at $x$.

We shall also require sufficient conditions on $k$ and $p$ that ensure there is a continuous embedding, $W^{k-2,p}(M;\RR^n) \subset (W^{k,p}(M;\RR^n))^*$ and, to this end, we have the

\begin{lem}[Continuous embedding of a Sobolev space into a dual space]
\label{lem:Continuous_embedding_Wk-2p_into_Wkpdual}
Let $d\geq 2$ and $k\geq 1$ be integers and $p\in [1,\infty)$ a constant such that
\[
kp>d \quad \text{or}\quad k=d \text{ and } p=1,
\]
and, in addition, that $p > 1$ if $k=2$. Then there is a continuous embedding,
\[
W^{k-2,p}(M;\RR) \subset (W^{k,p}(M;\RR))^*
\]
\end{lem}

We give the proof of Lemma \ref{lem:Continuous_embedding_Wk-2p_into_Wkpdual} in Appendix \ref{sec:Continuity_Sobolev_multiplication_maps}. We can now use Lemmas \ref{lem:subanalytic} and \ref{lem:Continuous_embedding_Wk-2p_into_Wkpdual} to establish real analyticity of the gradient map in the following analogue of Proposition \ref{prop:analyticharmd}, giving real analyticity of the energy functional.\footnote{One can also establish real analyticity of the gradient map directly by adapting the proof of Proposition \ref{prop:analyticharmd}.}

\begin{prop}[Analyticity of the gradient map]
\label{prop:Gradient_harmonic_map_operator_analyticity}
Assume the hypotheses of Lemma \ref{lem:Gradient_harmonic_map_operator_smooth} and, in addition, that $p > 1$ if $k=2$. If $(N,h)$ is real analytic and endowed with an isometric, real analytic embedding, $(N,h) \subset \RR^n$, then the gradient map \eqref{eq:Gradient_harmonic_map_operator} is real analytic,
\[
W^{k,p}(M;N) \ni f \mapsto \sM(f)
\in
W^{k-2,p}(M;\RR^n),
\]
where $\sM(f) \in W^{k-2,p}(M;f^*TN)$.
\end{prop}

\begin{proof}
Proposition \ref{prop:analyticharmd} implies that the map,
\[
W^{k,p}(M;N) \ni f \mapsto \sE'(f) \in (T_fW^{k,p}(M;N))^* = (W^{k,p}(M;f^*T^*N))^* \subset
(W^{k,p}(M;\RR^n))^*,
\]
is real analytic, while Lemma \ref{lem:Gradient_harmonic_map_operator_smooth} ensures that the gradient map,
\[
W^{k,p}(M;N) \ni f \mapsto \sM(f) \in W^{k-2,p}(M;f^*TN) \subset
W^{k-2,p}(M;\RR^n),
\]
is $C^\infty$. But Lemma \ref{lem:Continuous_embedding_Wk-2p_into_Wkpdual} yields  a continuous embedding,
\[
W^{k-2,p}(M;\RR^n) \subset (W^{k,p}(M;\RR^n))^*,
\]
and therefore analyticity of $\sM$ follows from Lemma \ref{lem:subanalytic}.
\end{proof}

The Hessian of $\sE$ at $f \in W^{k,p}(M;N)$ may be defined by
\begin{align*}
\sE''(f)(v,w)
&:=
\left.\frac{\partial^2}{\partial s\partial t}\sE(\pi_h(f + sv + tw))\right|_{s=t=0}
\\
&\,=
\left.\frac{d}{dt}\sE'(\pi_h(f+tw))(v)\right|_{t=0}
\\
&\,= (w, \sM'(f)(v))_{L^2(M;f^*TN)},
\end{align*}
for all $v, w \in W^{k,p}(M;f^*TN)$. The preceding general definition yields several equivalent expressions for the Hessian, $\sE''(f)$, and Hessian operator, $\sM'(f)$. One finds that \cite[Equation (4.3)]{KwonThesis}
\begin{equation}
\label{eq:harmonichessian_Kwon_thesis}
\sM'(f)v = \Delta_g v  - 2A_h(f)(df,dv) - (dA_h)(v)(df,df),
\quad\forall\, v \in W^{k,p}(M;f^*TN).
\end{equation}
Alternatively, from Lemma \ref{lem:Euler-Lagrange_equation_and_gradient_harmonic_map_energy_functional} and its proof, we have for all $v, w \in W^{k,p}(M;f^*TN)$,
\begin{align*}
\sM'(f)v &= \left.\frac{d}{dt} (d\pi_h(f+tv)w, \Delta_g \pi_h(f+tv))_{L^2(M)} \right|_{t=0}
\\
&= (d\pi_h(f)w, \Delta_g d\pi_h(f)(v))_{L^2(M)} + (d^2\pi_h(f)(v,w), \Delta_g f)_{L^2(M)}
\\
&= (d\pi_h(f)w, \Delta_g v)_{L^2(M)} + (d^2\pi_h(f)(v,w), \Delta_g f)_{L^2(M)},
\end{align*}
and thus
\begin{equation}
\label{eq:harmonichessian}
\sM'(f)v = d\pi_h(f)\Delta_gv + d^2\pi_h(f)(v,\cdot)^*\Delta_g f, \quad\forall\, v \in W^{k,p}(M;f^*TN).
\end{equation}
Before proceeding further, we shall need to consider the dependence of the operators, $d\pi_h(f)$ and $d^2\pi_h(f)$, on the maps, $f$. By \cite[Section 2.12.3, Theorem 1, Equation (v)]{Simon_1996}, we see that $d^2\pi_h(y)(v,w) = - A_h(y)(v,w)$ for every $y\in N$ and $v,w\in T_yN$. Therefore,
\begin{equation}
\label{eq:Hessian_nearest-point_projection}
d^2\pi_h(\tilde f)(v,w) = -A_h(\tilde f)(v,w) \in C^\infty(M; \tilde f^*(TN)^\perp), \quad\forall\, v,w \in C^\infty(M; \tilde f^*TN).
\end{equation}
We observe from the expression \eqref{eq:2ndfundamentalform} that the operator, $A_h(y):T_yN\times T_yN \to (T_yN)^\perp$, extends to an operator, $A_h(y):\RR^n\times \RR^n \to \RR^n$, for all $y \in N$.

Let $\cE, \cF$ be Banach spaces and $\Fred(\cE,\cF) \subset \sL(\cE,\cF)$ denote the subset of Fredholm operators.

\begin{thm}[Openness of the subset of Fredholm operators]
(See \cite[Corollary 19.1.6]{Hormander_v3}.)
\label{thm:Openness_set_Fredholm_operators}
The subset, $\Fred(\cE,\cF) \subset \sL(\cE,\cF)$, is open, the function, $\Fred(\cE,\cF) \ni T \mapsto \dim\Ker T$ is upper semi-continuous, and $\Ind T$ is constant in each connected component of $\Fred(\cE,\cF)$.
\end{thm}

In particular, given $T \in \Fred(\cE,\cF)$, there exists $\eps = \eps(T) \in (0,1]$ such that if $S \in \sL(\cE,\cF)$ obeys $\|S-T\|_{\sL(\cE,\cF)} < \eps$, then $S \in \Fred(\cE,\cF)$ and $\Ind S = \Ind T$. We can now prove that the Hessian operator, $\sM'(f)$, is Fredholm with index zero.

\begin{prop}[Fredholm and index zero properties of the Hessian operator for the harmonic map energy functional]
\label{prop:fredharmd}
Let $d\geq 2$ and $k \geq 1$ be integers and $p\in (1,\infty)$ be such that $kp > d$. Let $(M,g)$ and $(N,h)$ be closed, smooth Riemannian manifolds, with $M$ of dimension $d$ and $(N,h) \subset \RR^n$ a $C^\infty$ isometric embedding.
%PF-COMMENT-9-10-2016 We do not need f to be weakly harmonic
If $f \in W^{k,p}(M;N)$, then
\[
\sM'(f): W^{k,p}(M;f^*TN) \to W^{k-2,p}(M;f^*TN),
\]
is a Fredholm operator with index zero.
\end{prop}

\begin{rmk}[Further applications]
The proof of Proposition \ref{prop:fredharmd} could be adapted to extend \cite[Lemma 41.1]{Feehan_yang_mills_gradient_flow_v4} and \cite[Theorem A.1]{Feehan_Maridakis_Lojasiewicz-Simon_coupled_Yang-Mills_v6} from the case of elliptic partial differential operators with $C^\infty$ coefficients to those with suitable Sobolev coefficients.
\end{rmk}

\begin{proof}[Proof of Proposition \ref{prop:fredharmd}]
If $\tilde f\in C^\infty(M;N)$, then $\tilde f^*TN$ is a $C^\infty$ vector bundle and the conclusion is an immediate consequence of \cite[Lemma 41.1]{Feehan_yang_mills_gradient_flow_v4} or \cite[Theorem A.1]{Feehan_Maridakis_Lojasiewicz-Simon_coupled_Yang-Mills_v6}, since the expression \eqref{eq:harmonichessian_Kwon_thesis} tells us that
\[
\sM'(\tilde f):C^\infty(M;\tilde f^*TN) \to C^\infty(M;\tilde f^*TN)
\]
is an elliptic, linear, second-order partial differential operator with $C^\infty$ coefficients and $\sM'(\tilde f) - \sM'(\tilde f)^*$ is a first-order partial differential operator.

For the remainder of the proof, we focus on the case of maps in $W^{k,p}(M;N)$. Since $C^\infty(M;N)$ is dense in $W^{k,p}(M;N)$, the space $W^{k,p}(M;N)$ has an open cover consisting of $W^{k,p}(M;N)$-open balls centered at maps in $C^\infty(M;N)$. Given $\tilde f\in C^\infty(M;N)$, then for all $f \in W^{k,p}(M;N)$ that are $W^{k,p}(M;N)$-close enough to $\tilde f$, Lemma \ref{lem:Isomorphism_Sobolev_spaces_sections_two_nearby_maps} provides isomorphisms of Banach spaces,
\begin{align*}
d\pi_h(f):W^{k,p}(M;\tilde f^*TN) &\cong W^{k,p}(M;f^*TN),
\\
d\pi_h(\tilde f):W^{k-2,p}(M;f^*TN) &\cong W^{k-2,p}(M;\tilde f^*TN).
\end{align*}
The composition of a Fredholm operator with index zero and two invertible operators is a Fredholm operator with index zero and so the composition,
\[
d\pi_h(\tilde f)\circ\sM'(f)\circ d\pi_h(f): W^{k,p}(M;\tilde f^*TN) \to W^{k-2,p}(M;\tilde f^*TN),
\]
is a Fredholm operator with index zero if and only if
\[
\sM'(f): W^{k,p}(M;f^*TN) \to W^{k-2,p}(M;f^*TN),
\]
is a Fredholm operator with index zero.

Given $\eps \in (0,1]$, we claim that there exists $\delta = \delta(\tilde f, g,h,k,p,\eps) \in (0,1]$ with the following significance. If $f\in W^{k,p}(M;N)$ obeys
\[
\|f - \tilde f\|_{W^{k,p}(M;\RR^n)} < \delta,
\]
then
\begin{equation}
\label{eq:Difference_Hessian_operators_two_nearby_maps}
\|d\pi_h(\tilde f)\circ\sM'(f)\circ d\pi_h(f) - \sM'(\tilde f)\|_{\sL(W^{k,p}(M;\tilde f^*TN), W^{k-2,p}(M;\tilde f^*TN))} < \eps.
\end{equation}
Assuming \eqref{eq:Difference_Hessian_operators_two_nearby_maps}, Theorem \ref{thm:Openness_set_Fredholm_operators} implies that $\sM'(f)$ is Fredholm with index zero, the desired conclusion for $f\in W^{k,p}(M;N)$. To prove \eqref{eq:Difference_Hessian_operators_two_nearby_maps}, it suffices to establish the following claims.

\begin{claim}[Continuity of the differential of the nearest-point projection map]
\label{claim:Differential_nearest-point_projection_continuous_function_of_f_in_Wkp_maps}
For $l=k$ or $k-2$, the following map is continuous,
\[
W^{k,p}(M;N) \ni f \mapsto d\pi_h(f) \in \sL\left(W^{l,p}(M;\RR^n), W^{l,p}(M;f^*TN)\right) \subset \sL\left(W^{l,p}(M;\RR^n)\right).
\]
\end{claim}

\begin{proof}[Proof of Claim \ref{claim:Differential_nearest-point_projection_continuous_function_of_f_in_Wkp_maps}]
By \eqref{eq:Sobolev_smooth_projection}, the map
\[
W^{k,p}(M;N) \ni f \mapsto d\pi_h(f) \in W^{k,p}(M;\End(\RR^n)),
\]
is smooth. Also, there is a continuous embedding, $W^{k,p}(M;\End(\RR^n))\subset \sL\left(W^{l,p}(M;\RR^n)\right)$. To see this, observe that the bilinear map,
\[
W^{k,p}(M;\End(\RR^n)) \times W^{l,p}(M;\RR^n)\ni (\alpha, \xi) \to \alpha(\xi) \in W^{l,p}(M;\RR^n),
\]
is continuous since \cite[Corollary 9.7 and Theorem 9.13]{PalaisFoundationGlobal} and the proof of Lemma \ref{lem:Isomorphism_Sobolev_spaces_sections_two_nearby_maps} imply that $W^{l,p}(M;\RR^n)$ is a continuous $W^{k,p}(M; \End(\RR^n))$-module when $|l| \leq k$. Thus,
%COMMENT-PF-9-10-2016 Recheck
\begin{align*}
\|\alpha\|_{\sL\left(W^{l,p}(M;\RR^n)\right)}
&= \sup_{\begin{subarray}{c}\xi \in W^{l,p}(M;\RR^n) \less\{0\} \\ \|\xi\|_{W^{l,p}(M;\RR^n)}=1\end{subarray}} \|\alpha(\xi)\|_{W^{l,p}(M;\RR^n)}
\\
&\leq C\|\alpha\|_{W^{l,p}(M;\End(\RR^n))}.
\end{align*}
The conclusion follows.
\end{proof}

\begin{claim}[Continuity of the Hessian of the nearest-point projection map]
\label{claim:Hessian_nearest-point_projection_continuous_function_of_f_in_Wkp_maps}
The following map is continuous,
\begin{multline*}
W^{k,p}(M;N) \ni f \mapsto d^2\pi_h(f) \in \sL\left(W^{k,p}(M;f^*TN)\times W^{k,p}(M;f^*TN), W^{k,p}(M;f^*(TN)^\perp)\right)
\\
\subset \sL\left(W^{k,p}(M;\RR^n)\times W^{k,p}(M;\RR^n), W^{k,p}(M;\RR^n)\right).
\end{multline*}
\end{claim}

\begin{proof}[Proof of Claim \ref{claim:Hessian_nearest-point_projection_continuous_function_of_f_in_Wkp_maps}]
From the proof of \eqref{eq:Sobolev_smooth_projection}, the map
\[
W^{k,p}(M;N) \ni f \mapsto d^2\pi_h(f) \in W^{k,p}(M;\Hom(\RR^n\otimes \RR^n; \RR^n)),
\]
is smooth. Also, by an argument similar to that used in the proof of Claim \ref{claim:Differential_nearest-point_projection_continuous_function_of_f_in_Wkp_maps}, there is a continuous embedding,
\[
 W^{k,p}(M;\Hom(\RR^n\otimes \RR^n; \RR^n)) \subset \sL\left(W^{k,p}(M;\RR^n)\times W^{k,p}(M;\RR^n), W^{k,p}(M;\RR^n)\right).
\]
The conclusion follows by composing the two maps.
\end{proof}

\begin{claim}[Continuity of the Hessian of the energy function]
\label{claim:Hessian_energy_continuous_function_of_f_in_Wkp_maps}
The following map is continuous,
\begin{multline*}
W^{k,p}(M;N) \ni f \mapsto \sM'(f) \in \sL\left(W^{k,p}(M;f^*TN), W^{k-2,p}(M;f^*TN)\right)
\\
\subset \sL\left(W^{k,p}(M;\RR^n), W^{k-2,p}(M;\RR^n)\right).
\end{multline*}
\end{claim}

\begin{proof}[Proof of Claim \ref{claim:Hessian_energy_continuous_function_of_f_in_Wkp_maps}]
The conclusion follows from the expression \eqref{eq:harmonichessian} for $\sM'(f)$, the fact that the Sobolev space,
%COMMENT-PF-9-11-2016 Be more precise as to the continuous module properties being used
$W^{k-2,p}(M;\RR)$, is a continuous $W^{k,p}(M;\RR)$-module (see the proof of Claim \ref{claim:Differential_nearest-point_projection_continuous_function_of_f_in_Wkp_maps}), and Claims
\ref{claim:Differential_nearest-point_projection_continuous_function_of_f_in_Wkp_maps} and \ref{claim:Hessian_nearest-point_projection_continuous_function_of_f_in_Wkp_maps}.
\end{proof}

But the inequality \eqref{eq:Difference_Hessian_operators_two_nearby_maps} now follows from Claims \ref{claim:Differential_nearest-point_projection_continuous_function_of_f_in_Wkp_maps} and \ref{claim:Hessian_energy_continuous_function_of_f_in_Wkp_maps} and this completes the proof of Proposition \ref{prop:fredharmd}.
\end{proof}

In Lemma \ref{lem:Euler-Lagrange_equation_and_gradient_harmonic_map_energy_functional}, we computed the gradient, $\sM(f)$, of $\sE:W^{1,2}(M;N)\cap C(M;N)\to\RR$ at a map $f$ with respect to the inner product on $L^2(M;f^*TN)$. However, in order to apply Theorem~\ref{mainthm:Lojasiewicz-Simon_gradient_inequality}, we shall instead need to compute the gradient of $\sE:W^{k,p}(M;N)\to\RR$ with respect to the inner product on the Hilbert space, $L^2(M;f_\infty^*TN)$, defined by a \emph{fixed} map $f_\infty$. For this purpose, we shall need the forthcoming generalization of Remark \ref{rmk:chart_properties}.

\begin{lem}[Isomorphism of Sobolev spaces of sections defined by two nearby maps]
\label{lem:Isomorphism_Sobolev_spaces_sections_two_nearby_maps}
Let $d,k$ be integers and $p$ a constant obeying the hypotheses of Proposition \ref{prop:Banach_manifold}. Let $(M,g)$ and $(N,h)$ be closed, smooth Riemannian manifolds, with $M$ of dimension $d$ and $(N,h) \subset \RR^n$ a $C^\infty$ isometric embedding. Then there is a constant $\eps = \eps(g,h,k,p) \in (0,1]$ with the following significance. If $f, f_\infty \in W^{k,p}(M;N)$ obey $\|f-f_\infty\|_{W^{k,p}(M)} < \eps$ and $l\in\ZZ$ is an integer such that $|l| \leq k$, then
\begin{equation}
\label{eq:Isomorphism_Sobolev_spaces_sections_two_nearby_maps}
W^{l,p}(M;f^*TN) \ni v \mapsto d\pi_h(f_\infty)(v) \in W^{l,p}(M;f_\infty^*TN)
\end{equation}
is an isomorphism of Banach spaces that reduces to the identity at $f=f_\infty$.
\end{lem}

\begin{proof}
When $l=k$, the conclusion is provided by Remark \ref{rmk:chart_properties}. In general, observe that
\[
\|d\pi_h(f_\infty)(v)\|_{W^{l,p}(M;f_\infty^*TN)}
\leq
\|d\pi_h(f_\infty)\|_{\sL(W^{l,p}(M;f^*TN), W^{l,p}(M;f_\infty^*TN))} \|v\|_{W^{l,p}(M;f^*TN)}.
\]
We recall from Section \ref{subsec:Real_analytic_manifold_structure_Sobolev_space_maps} that the nearest point projection, $\pi_h:\RR^n \supset \sO \to N$, is a $C^\infty$ map on a normal tubular neighborhood of $N \subset \RR^n$ and that $d\pi_h:N\times\RR^n \to TN$ is $C^\infty$ orthogonal projection. In particular, $d\pi_h \in C^\infty(N; \Hom(N\times\RR^n, TN))$, while $f_\infty \in W^{k,p}(M;N)$ and thus $d\pi_h(f_\infty) \in W^{k,p}(M; \Hom(M\times\RR^n, f_\infty^*TN))$ by \cite[Lemma 9.9]{PalaisFoundationGlobal}.

The projection, $d\pi_h(f_\infty) \in W^{k,p}(M; \Hom(M\times\RR^n, f_\infty^*TN))$, acts on $v \in W^{l,p}(M;f^*TN)$ by pointwise inner product with coefficients in $W^{k,p}(M;\RR)$. By \cite[Corollary 9.7]{PalaisFoundationGlobal}, the Sobolev space, $W^{l,p}(M;\RR)$, is a continuous $W^{k,p}(M;\RR)$-module when $0 \leq l \leq k$, while \cite[Theorem 9.13]{PalaisFoundationGlobal}, implies that $W^{l,p'}(M;\RR)$ is a continuous $W^{k,p}(M;\RR)$-module when $-k \leq l \leq 0$ and $p' = p/(p-1) \in (1,\infty]$ is the dual H\"older exponent.

Moreover, if $f_1 \in W^{k,p}(M;\RR)$ and $\alpha \in W^{-k,p}(M;\RR) = (W^{k,p'}(M;\RR))^*$ and noting that $f_1\alpha \in (W^{k,p'}(M;\RR))^*$ and $f_1\alpha(f_2) \in \RR$ for $f_2 \in W^{k,p'}(M;\RR)$, then
\begin{align*}
\|f_1\alpha\|_{W^{-k,p}(M)} &= \sup_{f_2 \in W^{k,p'}(M;\RR)\less\{0\}} \frac{|f_1\alpha(f_2)|}{\|f_2\|_{W^{k,p'}(M)}}
\\
&\leq \|f_1\|_{C(M)}\sup_{f_2 \in W^{k,p'}(M;\RR)\less\{0\}} \frac{|(\alpha(f_2)|}{\|f_2\|_{W^{k,p'}(M)}}
\\
&\leq C\|f_1\|_{W^{k,p}(M)}\sup_{f_2 \in W^{k,p'}(M;\RR)\less\{0\}} \frac{\|\alpha\|_{(W^{k,p'}(M;\RR))^*}\|f_2\|_{W^{k,p'}(M)} }{\|f_2\|_{W^{k,p'}(M)}}
\\
&= C\|f_1\|_{W^{k,p}(M)} \|\alpha\|_{W^{-k,p}(M)} ,
\end{align*}
where $C = C(g,h,k,p) \in [1,\infty)$ is the norm of the continuous Sobolev embedding, $W^{k,p}(M;\RR) \subset C(M;\RR)$. Hence, $W^{l,p}(M;\RR)$ is also a continuous $W^{k,p}(M;\RR)$-module when $-k \leq l \leq 0$ and thus for all $l\in\ZZ$ such that $|l| \leq k$.

Consequently, the isomorphism,
\[
W^{k,p}(M;f^*TN) \ni v \mapsto d\pi_h(f_\infty)(v) \in W^{k,p}(M;f_\infty^*TN)
\]
extends to an isomorphism \eqref{eq:Isomorphism_Sobolev_spaces_sections_two_nearby_maps}, as claimed.
\end{proof}

Arguing as in the proof of Lemma \ref{lem:Euler-Lagrange_equation_and_gradient_harmonic_map_energy_functional} yields

\begin{lem}[Gradient of the harmonic map energy functional with respect to the $L^2$ metric defined by a fixed map]
\label{lem:Gradient_harmonic_map_energy_functional_fixed_map}
Assume the hypotheses of Lemma \ref{lem:Isomorphism_Sobolev_spaces_sections_two_nearby_maps}. Then the gradient of $\sE \circ \Phi_{f_\infty}$ at $u \in \sU_{f_\infty} \subset W^{k,p}(M;f_\infty^*TN)$ with respect to the inner product on $L^2(M;f_\infty^*TN)$,
\[
(\sE \circ \Phi_{f_\infty})'(u) = (u, \sM_{f_\infty}(f))_{L^2(M)}
\]
where $f = \Phi_{f_\infty}(u) = \pi_h(f_\infty+u) \in W^{k,p}(M;N)$, is given by
\begin{align*}
\sM_{f_\infty}(f) &:= d\pi_h(f_\infty)d\pi_h(f)\Delta_g f
\\
&\,= d\pi_h(f_\infty)\sM(f)
\\
&\,= d\pi_h(f_\infty)(\Delta_g f - A_h(f)(df,df)) \in W^{k-2,p}(M;f_\infty^*TN),
\end{align*}
and $\sM_{f_\infty}(f) = \sM(f_\infty)$ at $f=f_\infty$.
\end{lem}

\begin{proof}
Using the Chain Rule, we calculate
\begin{align*}
(\sE\circ\Phi_{f_\infty})'(u) &= \sE'(\Phi_{f_\infty}(u)) d\Phi_{f_\infty}(u)
\\
&= \sE'(f) d\pi_h(f_\infty)(u)
\\
&= (d\pi_h(f_\infty)(u), \sM(f))_{L^2(M)}
\\
&= (u, d\pi_h(f_\infty)\sM(f))_{L^2(M)},
\end{align*}
noting that the pointwise orthogonal projection, $d\pi_h(f_\infty) \in \End(\RR^n)$, is self-adjoint. Since $\sM(f) = d\pi_h(f)\Delta_g f$ by Lemma \ref{lem:Euler-Lagrange_equation_and_gradient_harmonic_map_energy_functional}, this yields the claimed formula for $\sM_{f_\infty}(f)$.
\end{proof}

We are now ready to complete the

\begin{proof}[Proof of Theorem \ref{mainthm:Lojasiewicz-Simon_gradient_inequality_energy_functional_Riemannian_manifolds}]
By Remark~\ref{rmk:chart_properties}, there is a constant $C_4 = C_4(f,g,h,k,p) \in [1,\infty)$ such that for every $u\in \sU_{f_\infty} \subset W^{k,p}(M;f_\infty^*TN)$ and $f = \Phi_{f_\infty}(u) = \pi_h(f_\infty+u) \in W^{k,p}(M; N)$, we have
\begin{equation}
\label{eq:Kwon_4-7_f_infty}
C_4^{-1}\|f - f_\infty\|_{W^{k,p}(M;\RR^n)} \leq \|u\|_{W^{k,p}(M;f_\infty^*TN)}
\leq C_4 \|f - f_\infty\|_{W^{k,p}(M;\RR^n)}.
\end{equation}
We shall first derive a {\L}ojasiewicz--Simon gradient inequality for the function,
\[
\sE\circ\Phi_{f_\infty}: W^{k,p}(M;f_\infty^*TN) \supset \sU_{f_\infty} \to \RR,
\]
with gradient operator,
\[
\sM_{f_\infty}\circ\Phi_{f_\infty}:W^{k,p}(M;f_\infty^*TN) \supset \sU_{f_\infty} \to W^{k-2,p}(M;f_\infty^*TN).
\]
Note that the proof of Lemma \ref{lem:Isomorphism_Sobolev_spaces_sections_two_nearby_maps} implies that
\[
d\pi_h(f_\infty):W^{k-2,p}(M;\RR^n) \to W^{k-2,p}(M;f_\infty^*TN)
\]
is a bounded, linear operator. Lemmas \ref{lem:Continuous_embedding_Wk-2p_into_Wkpdual}, Proposition \ref{prop:fredharmd}, and --- since $(N,h)$ is real analytic --- Proposition~\ref{prop:Gradient_harmonic_map_operator_analyticity} ensure that the hypotheses of Theorem~\ref{mainthm:Lojasiewicz-Simon_gradient_inequality}  are fulfilled by choosing $x_\infty := f_\infty$ and
\[
\sX := W^{k,p}(M;f_\infty^*TN) \subset \tilde\sX := W^{k-2,p}(M;f_\infty^*TN) \subset \sX^* = W^{-k,p'}(M;f_\infty^*TN),
\]
noting that $\Phi_{f_\infty}(0) = f_\infty$, so $\sE\circ\Phi_{f_\infty}$ has a critical point at the origin in $W^{k,p}(M;f_\infty^*TN)$. Hence, there exist constants $\theta\in [1/2,1)$, and $\sigma_0 \in (0,\delta]$, and $Z_0 \in (0,\infty)$ (where $\delta \in (0,1]$ is the constant in \eqref{eq:Ball_in_Wkp_Vf} that defines the open neighborhood $\sU_{f_\infty}$ of the origin) such that for every $u \in W^{k,p}(M;f_\infty^*TN)$ obeying $\|u\|_{W^{k,p}(M;f_\infty^*TN)} < \sigma_0$ we have
\[
|(\sE\circ\Phi_{f_\infty})(u) - (\sE\circ\Phi_{f_\infty})(0)|^\theta
\leq
Z_0\|\sM_{f_\infty}\circ\Phi_{f_\infty}(u)\|_{W^{k-2,p}(M;f_\infty^*TN)}.
\]
If $f =\Phi_{f_\infty}(u) \in W^{k,p}(M; N)$ obeys $\|f_\infty-f\|_{W^{k,p}(M;\RR^n)}< C_4^{-1}\sigma_0$, then inequality \eqref{eq:Kwon_4-7_f_infty} implies that $\|u\|_{W^{k,p}(M;f_\infty^*TN)} < \sigma_0$. Moreover,
\[
(\sM_{f_\infty}\circ\Phi_{f_\infty})(u) = d\pi_h(f_\infty)\circ\sM(\Phi_{f_\infty}(u)) = d\pi_h(f_\infty)\circ\sM(f),
\]
by Lemma \ref{lem:Gradient_harmonic_map_energy_functional_fixed_map} and Lemma \ref{lem:Isomorphism_Sobolev_spaces_sections_two_nearby_maps} implies that
\[
\|d\pi_h(f_\infty)\circ\sM(f)\|_{W^{k-2,p}(M;f_\infty^*TN)}
\leq
C\|\sM(f)\|_{W^{k-2,p}(M;f^*TN)},
\]
for a constant $C = C(f_\infty,g,h,k,p) \in [1,\infty)$. Therefore,
\[
|\sE(f) - \sE(f_\infty)|^\theta
\leq
CZ\|\sM(f)\|_{W^{k-2,p}(M;f^*TN)}.
\]
This yields inequality \eqref{eq:Lojasiewicz-Simon_gradient_inequality_harmonic_map_energy_functional_Riemannian_manifold} for constants $Z = CZ_0$ and $\sigma = C_4^{-1}\sigma_0$ and concludes the proof of Theorem~\ref{mainthm:Lojasiewicz-Simon_gradient_inequality_energy_functional_Riemannian_manifolds}.

The proof that the optimal {\L}ojasiewicz--Simon gradient inequality \eqref{eq:Lojasiewicz-Simon_gradient_inequality_harmonic_map_energy_functional_Riemannian_manifold} holds with $\theta=1/2$ under the condition \eqref{eq:Lojasiewicz-Simon_gradient_inequality_harmonic_map_neighborhood_Riemannian_manifold}
now follows \mutatis the proof of the inequality with $\theta\in[1/2,1)$ in the real analytic case with the aid of Theorem \ref{mainthm:Optimal_Lojasiewicz-Simon_gradient_inequality_Morse-Bott_energy_functional}. This concludes the proof of Theorem~\ref{mainthm:Lojasiewicz-Simon_gradient_inequality_energy_functional_Riemannian_manifolds}.
\end{proof}

\subsection{Application to the $L^2$ {\L}ojasiewicz--Simon gradient inequality for the harmonic map energy function}
\label{subsec:Application_Lojasiewicz-Simon_gradient_inequality_harmonic_map_energy_functional_L2}

Before proceeding to the proof of Corollary \ref{maincor:Lojasiewicz-Simon_gradient_inequality_energy_functional_Riemannian_manifolds_L2}, we shall need the following two technical lemmas.

\begin{lem}[Continuity of Sobolev multiplication maps]
\label{lem:Continuity_Sobolev_multiplication_maps}
Let $d\geq 2$ and $k \geq 2$ be integers and $p\in [2,\infty)$ be such that $kp > d$. Let $(M,g)$ be a closed, smooth Riemannian manifold of dimension $d$. Then the following Sobolev multiplication maps are continuous:
\begin{align}
\label{eq:Continuity_Sobolev_multiplication_maps_Wkp_times_L2_to_L2}
W^{k,p}(M; \RR) \times L^2(M;\RR) &\to L^2(M;\RR),
\\
\label{eq:Continuity_Sobolev_multiplication_maps_Wk-2p_times_W22_to_L2}
W^{k-2,p}(M; \RR) \times W^{2,2}(M;\RR) &\to L^2(M;\RR),
\\
\label{eq:Continuity_Sobolev_multiplication_maps_Wkp_times_W22_to_L2}
W^{k,p}(M; \RR) \times W^{2,2}(M;\RR) &\to W^{2,2}(M;\RR).
\end{align}
\end{lem}

The proof of Lemma \ref{lem:Continuity_Sobolev_multiplication_maps} is quite technical, so we shall provide that in Appendix \ref{sec:Continuity_Sobolev_multiplication_maps}. We have the following analogue of Lemma~\ref{lem:Isomorphism_Sobolev_spaces_sections_two_nearby_maps}.

\begin{lem}[Isomorphism of Sobolev spaces of sections defined by two nearby maps]
\label{lem:L2is}
Let $d\geq 2$, $k\geq 2$ be integers and $p \in [2,\infty)$ a constant such that $kp>d$. Let $(M,g)$ and $(N,h)$ be closed, smooth Riemannian manifolds, with $M$ of dimension $d$ and $(N,h) \subset \RR^n$ a $C^\infty$ isometric embedding. Then there is a constant $\eps = \eps(g,h,k,p) \in (0,1]$ with the following significance. If $f, f_\infty \in W^{k,p}(M;N)$ obey $\|f-f_\infty\|_{W^{k,p}(M)} < \eps$ and $l=0,2$, then
\begin{equation}
\label{eq:Isomorphism_Sobolev_spaces_sections_two_nearby_maps_Wl2}
W^{l,2}(M;f^*TN) \ni v \mapsto d\pi_h(f_\infty)(v) \in W^{l,2}(M;f_\infty^*TN)
\end{equation}
is an isomorphism of Banach spaces that reduces to the identity at $f=f_\infty$.
\end{lem}

\begin{proof}
We adapt \mutatis the proof of Lemma~\ref{lem:Isomorphism_Sobolev_spaces_sections_two_nearby_maps}, using the fact that $W^{2,2}(M;\RR)$ and $L^2(M;\RR)$ are continuous $W^{k,p}(M;\RR)$-modules by Lemma~\ref{lem:Continuity_Sobolev_multiplication_maps}, using the continuous Sobolev multiplication maps \eqref{eq:Continuity_Sobolev_multiplication_maps_Wkp_times_L2_to_L2} and \eqref{eq:Continuity_Sobolev_multiplication_maps_Wkp_times_W22_to_L2}.
\end{proof}

We can now proceed to the

\begin{proof}[Proof of Corollary \ref{maincor:Lojasiewicz-Simon_gradient_inequality_energy_functional_Riemannian_manifolds_L2}]
Consider Item \eqref{item:maincor_LS_grad_ineq_L2_d=2_k=1}. For $p \in (2,\infty)$, let $p' := p/(p-1) \in (1,2)$. Then \cite[Theorem 4.12]{AdamsFournier} implies that $W^{1,p'}(M;\RR) \subset L^2(M;\RR)$ is a continuous Sobolev embedding if $(p')^* = 2p'/(2-p') = 2p/(2(p-1)-p) \geq 2$, a condition that holds for all $p\in(1,\infty)$ since it is equivalent to $p \geq 2(p-1)-p = p-2$ or $0 \geq -2$. Since $kp>d$ by hypothesis and $d=2$ and $k=1$, then we must restrict $p$ to the range $2<p<\infty$. By density and duality, then $L^2(M;\RR) \subset W^{-1,p}(M;\RR)$ is a continuous Sobolev embedding. But inequality \eqref{eq:Lojasiewicz-Simon_gradient_inequality_harmonic_map_energy_functional_Riemannian_manifold} from Theorem \ref{mainthm:Lojasiewicz-Simon_gradient_inequality_energy_functional_Riemannian_manifolds} (with $d=2$, $k=1$, and $2<p<\infty$ yields
\[
\|\sM(f)\|_{W^{-1,p}(M;f^*TN)}
\geq
Z|\sE(f) - \sE(f_\infty)|^\theta,
\]
while, applying \eqref{eq:Lojasiewicz-Simon_gradient_inequality_harmonic_map_neighborhood_Riemannian_manifold} and Lemma \ref{lem:Isomorphism_Sobolev_spaces_sections_two_nearby_maps} to give equivalences of the norms on $W^{-1,p}(M;f^*TN)$ and $W^{-1,p}(M;f_\infty^*TN)$ and on $L^2(M;f^*TN)$ and $L^2(M;f_\infty^*TN)$,
\begin{align*}
\|\sM(f)\|_{W^{-1,p}(M;f^*TN)} &\leq C\|\sM(f)\|_{W^{-1,p}(M;f_\infty^*TN)}
\\
&\leq C\|\sM(f)\|_{L^2(M;f_\infty^*TN)} \quad\text{(by continuity of $L^2(M;\RR) \subset W^{-1,p}(M;\RR)$)}
\\
&\leq C\|\sM(f)\|_{L^2(M;f^*TN)},
\end{align*}
for $C=C(g,h,p,f_\infty) \in [1,\infty)$. Combining these inequalities yields Item \eqref{item:maincor_LS_grad_ineq_L2_d=2_k=1}.

Consider Item \eqref{item:maincor_LS_grad_ineq_L2_d=3_k=1}. For $p \in (3,\infty)$, let $p' := p/(p-1) \in (1,3/2)$. Then \cite[Theorem 4.12]{AdamsFournier} implies that $W^{1,p'}(M;\RR) \subset L^2(M;\RR)$ is a continuous Sobolev embedding if $(p')^* = 3p'/(3-p') = 3p/(3(p-1)-p) \geq 2$, a condition that is equivalent to $3p \geq 6(p-1)-2p = 4p-6$ or $p \leq 6$. Since $kp>d$ by hypothesis and $d=3$ and $k=1$, then we must restrict $p$ to the range  $3<p\leq 6$. The remainder of the argument for Item \eqref{item:maincor_LS_grad_ineq_L2_d=2_k=1} now applies unchanged to give Item \eqref{item:maincor_LS_grad_ineq_L2_d=3_k=1}.

Consider Item \eqref{item:maincor_LS_grad_ineq_L2_dgeq2_kgeq2}. We shall apply Theorem~\ref{mainthm:Lojasiewicz-Simon_gradient_inequality2} with the choices of Banach and Hilbert spaces,
\begin{multline*}
\sX := W^{k,p}(M;f_\infty^*TN),\quad \tilde\sX := W^{k-2,p}(M;f_\infty^*TN),
\\
\sG := W^{2,2}(M;f_\infty^*TN) \quad\text{and}\quad \tilde\sG := L^2(M; f_\infty^*TN).
\end{multline*}
Proposition \ref{prop:Banach_manifold} assures us that $\Phi_{f_\infty} = \pi_h(f_\infty+\cdot)$ is a $C^\infty$ (real analytic) inverse coordinate chart that gives a diffeomorphism from an open neighborhood of the origin, $\sU_{f_\infty} \subset W^{k,p}(M;f_\infty^*TN)$, onto an open neighborhood of $f_\infty$ in the $C^\infty$ (real analytic) Banach manifold, $W^{k,p}(M;N)$. We thus choose the energy function,
\[
\sE\circ\Phi_{f_\infty}: W^{k,p}(M;f_\infty^*TN) \supset \sU_{f_\infty} \to \RR,
\]
with its gradient map given by Lemma \ref{lem:Gradient_harmonic_map_energy_functional_fixed_map},
\[
\sM_{f_\infty}\circ\Phi_{f_\infty}:W^{k,p}(M;f_\infty^*TN) \supset \sU_{f_\infty} \ni u
\mapsto
\sM_{f_\infty}(\Phi_{f_\infty}(u)) \in W^{k-2,p}(M;f_\infty^*TN),
\]
with gradient $\sM_{f_\infty}(\Phi_{f_\infty}(u))$ related to the differential of $\sE\circ\Phi_{f_\infty}$ at $u \in \sU_{f_\infty}$ by
\[
(\sE\circ\Phi_{f_\infty})'(u)  = (u, \sM_{f_\infty}(\Phi_{f_\infty}(u)))_{L^2(M;f_\infty^*TN)},
\quad\forall\, u \in \sU_{f_\infty},
\]
where, for $f =\Phi_{f_\infty}(u) \in W^{k,p}(M;N)$ and $\sM(f) = d\pi_h(f)\Delta_g f \in W^{k-2,p}(M;f^*TN)$ as in \eqref{eq:Gradient_harmonic_map_operator},
\[
\sM_{f_\infty}(f) = d\pi_h(f_\infty)\sM(f) \in W^{k-2,p}(M;f_\infty^*TN).
\]
We shall need the following generalization of Claim~\ref{claim:Hessian_energy_continuous_function_of_f_in_Wkp_maps}.

\begin{claim}[Continuity of the Hessian of the energy function]
\label{claim:Hessian_energy_continuous_function_of_f_in_Wkp_maps_W22_to_L2operator}
For each $f \in W^{k,p}(M;N)$, the Hessian operator,
\[
\sM'(f) \in \sL\left(W^{k,p}(M;f^*TN), W^{k-2,p}(M;f^*TN)\right),
\]
given by \eqref{eq:harmonichessian}, namely
\[
W^{k,p}(M;f^*TN) \ni v \mapsto \sM'(f)v = d\pi_h(f)\Delta_gv + d^2\pi_h(f)(v,\cdot)^*\Delta_g f \in W^{k-2,p}(M;f^*TN),
\]
extends to a bounded linear operator,
\[
\sM_1(f) \in \sL\left(W^{2,2}(M;f^*TN), L^2(M;f^*TN)\right),
\]
and the following map is continuous,
\begin{equation}
\label{eq:Continuity_M1_Wkp_maps_to_bounded_operators_W22_to_L2}
\sM_1: W^{k,p}(M;N) \ni f \mapsto \sM_1(f) \in \sL\left(W^{2,2}(M;\RR^n), L^2(M;\RR^n)\right).
\end{equation}
\end{claim}

\begin{rmk}[Application of Claim \ref{claim:Hessian_energy_continuous_function_of_f_in_Wkp_maps_W22_to_L2operator} to $\sM_{f_\infty}$]
\label{rmk:Hessian_energy_continuous_function_of_u_in_Ufinfty_W22_to_L2operator}
From the definition of $\sM_{f_\infty}$ in Lemma \ref{lem:Gradient_harmonic_map_energy_functional_fixed_map}, we see that Claim \ref{claim:Hessian_energy_continuous_function_of_f_in_Wkp_maps_W22_to_L2operator} and boundedness of the projection operator, $d\pi_h(f_\infty)$, in the forthcoming \eqref{eq:dpi:Wkp_maps_to_bounded_operators_L2} ensures that each Hessian operator,
\[
\sM_{f_\infty}'(u) \in \sL\left(W^{k,p}(M;f_\infty^*TN), W^{k-2,p}(M;f_\infty^*TN)\right),
\quad\text{for } u \in \sU_{f_\infty} \subset W^{k,p}(M;f_\infty^*TN),
\]
extends to a bounded linear operator,
\[
\sM_{f_\infty,1}(u) \in \sL\left(W^{2,2}(M;f_\infty^*TN), L^2(M;f_\infty^*TN)\right),
\]
such that (as required for the application of Theorem~\ref{mainthm:Lojasiewicz-Simon_gradient_inequality2}) the following map is continuous,
\begin{equation}
\label{eq:Continuity_M1_Ufinfty_to_bounded_operators_W22_to_L2}
\sM_{f_\infty,1}: \sU_{f_\infty} \ni u \mapsto \sM_{f_\infty,1}(u) \in \sL\left(W^{2,2}(M;f_\infty^*TN), L^2(M;f_\infty^*TN)\right),
\end{equation}
by virtue of smoothness of the inverse coordinate chart, $\Phi_{f_\infty}$.
\end{rmk}

\begin{proof}[Proof of Claim \ref{claim:Hessian_energy_continuous_function_of_f_in_Wkp_maps_W22_to_L2operator}]
In the proof of Lemma \ref{lem:Gradient_harmonic_map_operator_smooth}, we verified smoothness of the map \eqref{eq:Sobolev_smooth_projection}, namely
\[
W^{k,p}(M;N) \ni f \mapsto d\pi(f) \in W^{k,p}(M;\End(\RR^n)).
\]
According to Lemma~\ref{lem:Continuity_Sobolev_multiplication_maps}, the Sobolev multiplication maps \eqref{eq:Continuity_Sobolev_multiplication_maps_Wkp_times_L2_to_L2} and \eqref{eq:Continuity_Sobolev_multiplication_maps_Wk-2p_times_W22_to_L2} are continuous and thus $L^2(M;\RR^n)$ and $W^{2,2}(M;\RR^n)$ are continuous $W^{k,p}(M;\End(\RR^n))$-modules. In the proof of Claim~\ref{claim:Differential_nearest-point_projection_continuous_function_of_f_in_Wkp_maps}, we showed that
\[
W^{k,p}(M;\End(\RR^n))\subset \sL\left(W^{l,p}(M;\RR^n)\right),
\]
is a continuous embedding for $l=k$ or $k-2$; this proof adapts \mutatis to give a continuous embedding for $l=2$ or $0$,
\[
W^{k,p}(M;\End(\RR^n))\subset \sL\left(W^{l,2}(M;\RR^n)\right).
\]
Hence, the following maps are continuous,
\begin{align}
\label{eq:dpi:Wkp_maps_to_bounded_operators_L2}
W^{k,p}(M;N) \ni f &\mapsto d\pi(f) \in \sL\left(L^2(M;\RR^n)\right),
\\
\label{eq:dpi:Wkp_maps_to_bounded_operators_W22}
W^{k,p}(M;N) \ni f &\mapsto d\pi(f) \in \sL(W^{2,2}(M;\RR^n)).
\end{align}
We have $\Delta_g \in \sL\left(W^{2,2}(M;\RR^n), L^2(M;\RR^n)\right)$ and so the following composition is continuous,
\begin{equation}
\label{eq:dpi_circ_Delta:Wkp_maps_to_bounded_operators_W22_to_L2}
W^{k,p}(M;N) \ni f \mapsto d\pi(f)\circ \Delta_g \in \sL\left(W^{2,2}(M;\RR^n), L^2(M;\RR^n)\right).
\end{equation}
By Claim \ref{claim:Hessian_nearest-point_projection_continuous_function_of_f_in_Wkp_maps}, the following map is smooth,
\[
W^{k,p}(M;N) \ni f \mapsto d^2\pi_h(f)(\cdot,\cdot)^* \in W^{k,p}(M;\Hom(\RR^n; \End(\RR^n)),
\]
and clearly the following linear map is also smooth,
\[
W^{k,p}(M;N) \ni f \mapsto \Delta_gf \in W^{k-2,p}(M;\RR^n).
\]
For $k\geq 2$, the \cite[Corollary 9.7]{PalaisFoundationGlobal} implies that the following multiplication map is continuous,
\[
W^{k,p}(M;\Hom(\RR^n; \End(\RR^n)) \times W^{k-2,p}(M;\RR^n) \to W^{k-2,p}(M; \End(\RR^n)).
\]
Therefore, the following composition is continuous,
\[
W^{k,p}(M;N) \ni f \mapsto d^2\pi_h(f)(\cdot,\cdot)^*\Delta_gf \in W^{k-2,p}(M; \End(\RR^n)).
\]
Using the continuity of the Sobolev multiplication map \eqref{eq:Continuity_Sobolev_multiplication_maps_Wk-2p_times_W22_to_L2} given by Lemma~\ref{lem:Continuity_Sobolev_multiplication_maps}, the verification of continuity of the embedding,
\[
W^{k,p}(M;\End(\RR^n))\subset \sL\left(W^{k-2,p}(M;\RR^n)\right),
\]
in the proof of Claim~\ref{claim:Differential_nearest-point_projection_continuous_function_of_f_in_Wkp_maps} adapts \mutatis to give a continuous embedding,
\[
W^{k-2,p}(M;\End(\RR^n))\subset \sL\left(W^{2,2}(M;\RR^n), L^2(M;\RR^n)\right).
\]
Hence, we see that the following composition is continuous,
\begin{equation}
\label{eq:d2pi_Delta:Wkp_maps_to_bounded_operators_W22_to_L2}
W^{k,p}(M;N) \ni f \mapsto d^2\pi_h(f)(\cdot,\cdot)^*\Delta_gf \in \sL\left(W^{2,2}(M;\RR^n), L^2(M;\RR^n)\right).
\end{equation}
Finally, the continuity of the maps \eqref{eq:dpi_circ_Delta:Wkp_maps_to_bounded_operators_W22_to_L2} and \eqref{eq:d2pi_Delta:Wkp_maps_to_bounded_operators_W22_to_L2} and the expression \eqref{eq:harmonichessian} for $\sM'(f)$ implies that the map,
\[
\sM' : W^{k,p}(M;N) \ni f \mapsto \sM'(f) \in \sL\left(W^{k,p}(M;\RR^n), W^{k-2,p}(M;\RR^n)\right)
\]
extends to give the continuous map \eqref{eq:Continuity_M1_Wkp_maps_to_bounded_operators_W22_to_L2}. This completes the proof of Claim \ref{claim:Hessian_energy_continuous_function_of_f_in_Wkp_maps_W22_to_L2operator}.
\end{proof}

Next we adapt the proof of Proposition \ref{prop:fredharmd} to prove the

\begin{claim}[Fredholm and index zero properties of the extended Hessian operator for the harmonic map energy function]
\label{claim:fredharmd_W22_to_L2}
For every $f\in W^{k,p}(M;N)$, the following operator has index zero,
\[
\sM_1(f) \in \sL\left(W^{2,2}(M; f^*TN), L^2(M; f^*TN)\right).
\]
\end{claim}

\begin{proof}[Proof of Claim \ref{claim:fredharmd_W22_to_L2}]
For any $\tilde f\in C^\infty(M;N)$, either \cite[Lemma 41.1]{Feehan_yang_mills_gradient_flow_v4} or \cite[Theorem A.1]{Feehan_Maridakis_Lojasiewicz-Simon_coupled_Yang-Mills_v6} implies that $\Ind \sM_1(\tilde f) = 0$. Moreover, we may approximate any Sobolev map $f\in W^{k,p}(M;N)$ by a smooth map $\tilde f\in C^\infty(M;N)$. Lemma~\ref{lem:L2is} implies that the operators,
\begin{align*}
d\pi_h(f):W^{2,2}(M;\tilde f^*TN) &\cong W^{2,2}(M;f^*TN),
\\
d\pi_h(\tilde f): L^2(M;f^*TN) &\cong L^2(M;\tilde f^*TN),
\end{align*}
are isomorphisms of Banach spaces whenever $f$ is $W^{k,p}(M;N)$-close enough to $\tilde f$. Hence, the composition,
\[
d\pi_h(\tilde f)\circ\sM_1(f)\circ d\pi_h(f): W^{2,2}(M;\tilde f^*TN) \to L^2(M;\tilde f^*TN),
\]
is a Fredholm operator with index zero if and only if
\[
\sM_1(f): W^{2,2}(M;f^*TN) \to L^2(M;f^*TN),
\]
is a Fredholm operator with index zero. But continuity of the maps \eqref{eq:Continuity_M1_Wkp_maps_to_bounded_operators_W22_to_L2}, \eqref{eq:dpi:Wkp_maps_to_bounded_operators_L2}, and \eqref{eq:dpi:Wkp_maps_to_bounded_operators_W22} implies that given $\eps \in (0,1]$, there exists $\delta = \delta(\tilde f, g,h, \eps) \in (0,1]$ with the following significance. If $f\in W^{k,p}(M;N)$ obeys
\[
\|f - \tilde f\|_{W^{k,p}(M;\RR^n)} < \delta,
\]
then
\[
\|d\pi_h(\tilde f)\circ\sM_1(f)\circ d\pi_h(f) - \sM_1(\tilde f)\|_{\sL(W^{2,2}(M;\tilde f^*TN), L^2(M;\tilde f^*TN))} < \eps.
\]
Theorem \ref{thm:Openness_set_Fredholm_operators} now implies that $\sM_1(f)$ is Fredholm with index zero, as desired for $f\in W^{k,p}(M;N)$. This completes the proof of Claim \ref{claim:fredharmd_W22_to_L2}.
\end{proof}

Following Remark \ref{rmk:Hessian_energy_continuous_function_of_u_in_Ufinfty_W22_to_L2operator}, we also have the

\begin{rmk}[Application of Claim \ref{claim:Hessian_energy_continuous_function_of_f_in_Wkp_maps_W22_to_L2operator} to the Hessian operator $\sM_{f_\infty}'$]
\label{rmk:fredharmd_finfty_W22_to_L2}
From the proof of Claim \ref{claim:fredharmd_W22_to_L2} and definition of $\sM_{1,f_\infty}(u)$ in Remark \ref{rmk:Hessian_energy_continuous_function_of_u_in_Ufinfty_W22_to_L2operator}, we also see that every $f_\infty \in W^{k,p}(M;N)$ and $u\in \sU_{f_\infty} \subset W^{k,p}(M;f_\infty^*N)$, the following operator has index zero,
\[
\sM_{1,f_\infty}(u) \in \sL\left(W^{2,2}(M; f_\infty^*TN), L^2(M; f_\infty^*TN)\right),
\]
as required for the application of Theorem~\ref{mainthm:Lojasiewicz-Simon_gradient_inequality2}.
\end{rmk}

The remainder of the proof of Theorem \ref{mainthm:Lojasiewicz-Simon_gradient_inequality_energy_functional_Riemannian_manifolds}
now adapts \mutatis to verify that the hypotheses of Theorem~\ref{mainthm:Lojasiewicz-Simon_gradient_inequality2} and Theorem \ref{mainthm:Optimal_Lojasiewicz-Simon_gradient_inequality_Morse-Bott_energy_functional} are obeyed when $\sE$ is real analytic or Morse--Bott, respectively. This completes the proof of Corollary \ref{mainthm:Lojasiewicz-Simon_gradient_inequality2}.
\end{proof}

\appendix

\section{Continuity of Sobolev embeddings and multiplication maps}
\label{sec:Continuity_Sobolev_multiplication_maps}
In this appendix, we first give the

\begin{proof}[Proof of Lemma \ref{lem:Continuous_embedding_Wk-2p_into_Wkpdual}]
Recall from \cite[Section 3.5--3.14]{AdamsFournier} that $(W^{k,p}(M;\RR))^* = W^{-k,p'}(M;\RR)$, where $p' \in (1,\infty]$ is the dual H\"older exponent defined by $1/p+1/p'=1$, so we must determine sufficient conditions on $k$ and $p$ that ensure continuity of the embedding, $W^{k-2,p}(M;\RR) \subset W^{-k,p'}(M;\RR)$.

Consider the case $k = 1$. Then $W^{-1,p}(M;\RR) \subset W^{-1,p'}(M;\RR)$ is a continuous embedding if and only if $p\geq p'$, that is $p \geq 2$ and the latter condition is assured by our hypothesis that $kp>d$ and $d\geq 2$.

Consider the case $k = 2$. Then $L^p(M;\RR) \subset W^{-2,p'}(M;\RR)$ is a continuous embedding and if $p>1$, it is the dual of a continuous embedding, $W^{2,p}(M;\RR) \subset L^{p'}(M;\RR)$, by \cite[Sections 3.5--3.14]{AdamsFournier}. According to \cite[Theorem 4.12]{AdamsFournier}, there is a continuous Sobolev embedding, $W^{2,p}(M;\RR) \subset C(M;\RR)$, by our hypothesis that $kp>d$, and hence the embedding, $W^{2,p}(M;\RR) \subset L^{p'}(M;\RR)$, is continuous, as required for this case.

Consider the case $k \geq 3$. According to \cite[Theorem 4.12]{AdamsFournier}, there are continuous Sobolev embeddings,
\begin{inparaenum}[\itshape a\upshape)]
\item $W^{k-2,p}(M;\RR) \subset L^{p^*}(M;\RR)$, if $(k-2)p < d$ and $p^* = dp/(d-(k-2)p)$, or
\item $W^{k-2,p}(M;\RR) \subset L^q(M;\RR)$, if $(k-2)p = d$ and $1 \leq q < \infty$, or
\item $W^{k-2,p}(M;\RR) \subset L^\infty(M;\RR)$, if $(k-2)p > d$.
\end{inparaenum}
By our hypothesis that $kp>d$, there is a continuous Sobolev embedding $W^{k,p}(M;\RR) \subset L^r(M;\RR)$ for any $r \in [1,\infty]$ by \cite[Theorem 4.12]{AdamsFournier} and hence, by duality, there is a continuous Sobolev embedding, $L^{r'}(M;\RR) \subset W^{-k,p'}(M;\RR)$ for any $r \in [1,\infty)$.

Consider the subcase $(k-2)p < d$. By choosing $r = p^*$, we obtain a continuous embedding
\[
W^{k-2,p}(M;\RR) \subset L^{p^*}(M;\RR) \subset W^{-k,p'}(M;\RR).
\]
Consider the subcases, $(k-2)p \geq d$. By choosing $q \in (1,\infty)$ and $r' = q$ for $r \in (1,\infty)$, we again obtain a continuous embedding
\[
W^{k-2,p}(M;\RR) \subset L^q(M;\RR) \subset W^{-k,p'}(M;\RR).
\]
This concludes the proof of Lemma \ref{lem:Continuous_embedding_Wk-2p_into_Wkpdual}.
\end{proof}

Next, we provide the

\begin{proof}[Proof of Lemma \ref{lem:Continuity_Sobolev_multiplication_maps}]
Continuity of the multiplication map \eqref{eq:Continuity_Sobolev_multiplication_maps_Wkp_times_L2_to_L2} is an immediate consequence of continuity of the Sobolev embedding, $W^{k,p}(M; \RR) \subset C(M;\RR)$, for $kp>d$ given by \cite[Theorem 4.12]{AdamsFournier}.

For \eqref{eq:Continuity_Sobolev_multiplication_maps_Wk-2p_times_W22_to_L2}, we shall apply Palais' \cite[Theorem 9.6]{PalaisFoundationGlobal} (see Case \ref{case:claim_multi_d_geq_5} in the proof of \eqref{eq:Continuity_Sobolev_multiplication_maps_Wkp_times_W22_to_L2} when $d\geq 5$ below for a detailed review of Palais' hypotheses). We define $s_1 := (d/p)-(k-2) = (d/p)-k+2<2$ and $s_2 := (d/2)-2\geq 0$ and $\sigma := d/2$. Notice that $s_1+s_2  = (d/p)-k+(d/2) < d/2 = \sigma <d$ and that $\sigma > \max\{s_1, s_2\}$, which covers the case $s_1,s_2 <0$. Hence, the hypotheses of \cite[Theorem 9.6]{PalaisFoundationGlobal} are obeyed except when $s_1=s_2=0$; however, the latter case is provided by \cite[Theorem 9.5 (2)]{PalaisFoundationGlobal}. This proves \eqref{eq:Continuity_Sobolev_multiplication_maps_Wk-2p_times_W22_to_L2}.

For \eqref{eq:Continuity_Sobolev_multiplication_maps_Wkp_times_W22_to_L2}, we observe that if $p=2$, then the multiplication map \eqref{eq:Continuity_Sobolev_multiplication_maps_Wkp_times_W22_to_L2} is continuous by \cite[Corollary 9.7]{PalaisFoundationGlobal} for any $d \geq 2$, since $kp>d$ by hypothesis. For $p>2$, we shall separately consider the cases $d=2,3$, $d\geq 5$, and $d=4$.

\setcounter{case}{0}
\begin{case}[$d=2,3$]
Recall that $W^{2,2}(M;\RR)$ is a Banach algebra by \cite[Theorem 4.39]{AdamsFournier} when $1\leq d<4$ and so the multiplication map \eqref{eq:Continuity_Sobolev_multiplication_maps_Wkp_times_W22_to_L2} is continuous for any $p \geq 2$ by continuity of the Sobolev embedding, $W^{k,p}(M; \RR) \subset W^{2,2}(M;\RR)$.
\end{case}

If $p > 2$, then one could appeal in part to \cite[Theorem 9.6]{PalaisFoundationGlobal}, but it is simpler to just verify the result directly. For $f_1 \in W^{k,p}(M; \RR)$ and $f_2 \in W^{2,2}(M;\RR)$, we have
\[
\nabla(f_1f_2) = (\nabla f_1)f_2 + f_1\nabla f_2
\text{ and }
\nabla^2(f_1f_2) = (\nabla^2 f_1)f_2 + 2\nabla f_1\cdot \nabla f_2 + f_1\nabla^2 f_2.
\]
Hence,
\begin{align*}
\|\nabla(f_1f_2)\|_{L^2(\RR)} &\leq \|(\nabla f_1)f_2\|_{L^2(\RR)} + \|f_1\nabla f_2\|_{L^2(\RR)}
\end{align*}

\begin{case}[$d \geq 5$]
\label{case:claim_multi_d_geq_5}
We shall apply \cite[Theorem 9.6]{PalaisFoundationGlobal}, which for $r=2$ asserts that the following multiplication map is continuous,
\[
W^{k_1,p_1}(M; \RR) \times W^{k_2,p_2}(M; \RR) \to W^{l,q}(M;\RR),
\]
provided
\begin{inparaenum}[\itshape a\upshape)]
\item $1 \leq p_1,p_2,q < \infty$; and
\item $k_1, k_2 \geq l$; and
\item $s_1+s_2 < d$, where $s_1,s_2\in\RR$ are defined by $k_i=:(d/p)-s_i$ for $i=1,2$; and
\item for $\sigma$ defined by $l=(d/q)-\sigma$, then
\begin{inparaenum}[\itshape i\upshape)]
\item $\sigma \geq s_1+s_2$ if $s_1,s_2>0$; or
\item $\sigma \geq s_1$ if $s_1>0$ and $s_2\leq 0$, with strict inequality if $s_2=0$; or
\item $\sigma \geq s_2$ if $s_2>0$ and $s_1\leq 0$, with strict inequality if $s_1=0$; or
\item $\sigma \geq \max\{s_1,s_2\}$ if $s_1, s_2 <0$, with strict inequality if $\max\{s_1,s_2\}$ is an integer.
\end{inparaenum}
\end{inparaenum}
We choose $k_1=k$, $p_1=p$ and $k_2 = 2$, $p_2=2$, and $l=2$, $q=2$. We have $s_1 = d/p-k$, so $s_1<0$ by hypothesis, and $s_2 = (d/2)-2$, so $s_2 > 0$, and $s_1+s_2 < (d/2)-2 < d$. We also have $\sigma = (d/2)- 2 = s_2$, so $\sigma > s_2$ for $s_2>0$, as required when $s_1<0$. Hence, the multiplication map \eqref{eq:Continuity_Sobolev_multiplication_maps_Wkp_times_W22_to_L2} is continuous for $d \geq 5$ by \cite[Theorem 9.6]{PalaisFoundationGlobal}.
\end{case}

\begin{case}[$d = 4$]
Palais' \cite[Theorem 9.6]{PalaisFoundationGlobal} does not apply directly to this borderline case since examination of the choices for $d\geq 5$ reveals that we would have $s_1<0$ but $s_2=0=\sigma$ and thus $0 = \sigma \not> \max\{s_1,s_2\}=0$.

Let $f_1\in W^{k,p}(M;\RR)$ and $f_2 \in W^{2,2}(M;\RR)$. It is convenient (although not strictly necessary if we appealed instead to Palais' more general \cite[Theorem 9.5]{PalaisFoundationGlobal}) to separately consider the cases $k - 4/p>2$ and $k - 4/p \leq 2$.

\emph{Assume $k-4/p>2$.}
In this case, we have a continuous Sobolev embedding, $W^{k,p}(M; \RR) \subset C^2(M;\RR)$, by \cite[Theorem 4.12]{AdamsFournier} since $(k-2)p = kp - 2p > 4$, and so
\[
\|f_1f_2\|_{W^{2,2}(M;\RR)} \leq \|f_1\|_{C^2(M;\RR)} \|f_2\|_{W^{2,2}(M;\RR)} \leq C\|f_1\|_{W^{k,p}(M;\RR)} \|f_2\|_{W^{2,2}(M;\RR)},
\]
for a constant $C=C(g,k,p) \in [1,\infty)$. This proves continuity of the multiplication map \eqref{eq:Continuity_Sobolev_multiplication_maps_Wkp_times_W22_to_L2} when $k-4/p>2$.

\emph{Assume $k-4/p \leq 2$.}
To prove continuity of the multiplication map \eqref{eq:Continuity_Sobolev_multiplication_maps_Wkp_times_W22_to_L2}, we must show that
\[
\|f_1f_2\|_{W^{2,2}(M;\RR)}
\leq C\|f_1\|_{W^{k,p}(M;\RR)} \|f_2\|_{W^{2,2}(M;\RR)},
\]
for a constant $C=C(g,k,p) \in [1,\infty)$. Thus, it suffices to show that the $L^2$ norm of each one of the following terms,
\begin{equation}
\label{eq:Terms_nabla_f1f2_and_nabla2_f1f2}
f_1f_2, \ f_1 \nabla f_2, \ f_1 \nabla^2f_2 \quad\text{and}\quad  (\nabla f_1)f_2, \ (\nabla f_1) \nabla f_2, \ (\nabla^2f_1)f_2,
\end{equation}
is bounded by $C\|f_1\|_{W^{k,p}(M;\RR)} \|f_2\|_{W^{2,2}(M;\RR)}$.

By hypothesis of Lemma~\ref{lem:Continuity_Sobolev_multiplication_maps}, we have $kp>d$, so $kp-4>0$. According to \cite[Theorem 4.12]{AdamsFournier}, we thus have a continuous Sobolev embedding, $W^{k,p}(M; \RR) \subset C(M;\RR)$, and so the $L^2$ norms of each member of the first group of products in \eqref{eq:Terms_nabla_f1f2_and_nabla2_f1f2} is bounded by $C\|f\|_{W^{k,p}(M;\RR)}\|g\|_{W^{2,2}(M;\RR)}$, for a constant $C=C(g,k,p) \in [1,\infty)$, as desired.

To bound the $L^2$ norms of each of the products in the second group of terms in \eqref{eq:Terms_nabla_f1f2_and_nabla2_f1f2}, we need continuity of the following Sobolev multiplication maps,
\begin{align}
\label{eq:Continuity_Sobolev_multiplication_maps_d=4_Wk-1p_times_W22_to_L2}
W^{k-1,p}(M;\RR) \times W^{2,2}(M;\RR) &\to L^2(M;\RR),
\\
\label{eq:Continuity_Sobolev_multiplication_maps_d=4_Wk-1p_times_W12_to_L2}
W^{k-1,p}(M;\RR) \times W^{1,2}(M;\RR) &\to L^2(M;\RR),
\\
\label{eq:Continuity_Sobolev_multiplication_maps_d=4_Wk-2p_times_W22_to_L2}
W^{k-2,p}(M;\RR) \times W^{2,2}(M;\RR) &\to L^2(M;\RR).
\end{align}
Continuity of the multiplication map \eqref{eq:Continuity_Sobolev_multiplication_maps_d=4_Wk-1p_times_W22_to_L2} follows from continuity of the multiplication map \eqref{eq:Continuity_Sobolev_multiplication_maps_d=4_Wk-2p_times_W22_to_L2} via continuity of the Sobolev embedding, $W^{k-1,p}(M;\RR) \subset W^{k-2,p}(M;\RR)$.

To prove continuity of \eqref{eq:Continuity_Sobolev_multiplication_maps_d=4_Wk-1p_times_W12_to_L2}, we apply \cite[Theorem 9.6]{PalaisFoundationGlobal} with $s_1 = (d/p)-(k-1) = (4/p)-k+1$, so $s_1 < 1$, and $s_2 = (d/2)-1 = 1 > 0$ and $\sigma = (d/2)-0 = 2$. Notice that $s_1 + s_2 = 4/p-k+2 < 2 = \sigma <4=d$ and that if $s_1 \leq 0$ then we still have $s_2 = 1 < 2 = \sigma$. Hence, the hypotheses of \cite[Theorem 9.6]{PalaisFoundationGlobal} are obeyed and this proves \eqref{eq:Continuity_Sobolev_multiplication_maps_d=4_Wk-1p_times_W12_to_L2}.

To prove continuity of \eqref{eq:Continuity_Sobolev_multiplication_maps_d=4_Wk-2p_times_W22_to_L2}, we apply \cite[Theorem 9.6]{PalaisFoundationGlobal} with $s_1 = d/p-(k-2) = 4/p-k+2 \geq 0$, so $0 \leq s_1 < 2$, and as before, $s_2 = 0$ and $\sigma =2$. Notice that $s_1+s_2 = s_1 <2= \sigma< 4=d$. If $s_1>0$, then the hypotheses of \cite[Theorem 9.6]{PalaisFoundationGlobal} are obeyed and this proves \eqref{eq:Continuity_Sobolev_multiplication_maps_d=4_Wk-2p_times_W22_to_L2} when $s_1>0$.

Palais' \cite[Theorem 9.6]{PalaisFoundationGlobal} does not apply when\footnote{The omission of this case appears to be just an oversight.} $s_1=s_2=0$, but we can apply his more general \cite[Theorem 9.5 (2)]{PalaisFoundationGlobal}, which does include the case $s_1=s_2=0$, using $d=4$, $l=0$, $q=2$ and observing that we obtain a strict inequality, $0 = l < (d/q)-\max\{s_1,s_2\} = 4/2-0 = 2$, as required for this case. Moreover, $k_1 = k-2 \geq l=0$ and $k_2 = 2 \geq l=0$. Hence, the hypotheses of \cite[Theorem 9.5 (2)]{PalaisFoundationGlobal} are obeyed when $s_1=s_2=0$ and this completes the proof of \eqref{eq:Continuity_Sobolev_multiplication_maps_d=4_Wk-2p_times_W22_to_L2}.
\end{case}

This concludes the proof of continuity of the multiplication map \eqref{eq:Continuity_Sobolev_multiplication_maps_Wkp_times_W22_to_L2} and therefore the proof of Lemma~\ref{lem:Continuity_Sobolev_multiplication_maps}.
\end{proof}

%%%%%%%%%%%%%%%%%%%%%%%%%%%%%%%%%%%%%%%%%%%%%%%%%%%%%%%%%%%%%%%%%%%%%%%%%%%%%%%
%
%                                bibliography
%
%%%%%%%%%%%%%%%%%%%%%%%%%%%%%%%%%%%%%%%%%%%%%%%%%%%%%%%%%%%%%%%%%%%%%%%%%%%%%%%

\bibliography{master,mfpde}
%\bibliography{/Users/pfeehan/Dropbox/LATEX/Bibinputs/master,/Users/pfeehan/Dropbox/LATEX/Bibinputs/mfpde}
\bibliographystyle{amsplain}

\end{document}

%% file: latexformat.tex
\newtheorem{thm}{Theorem}[section]
\newtheorem*{thm*}{Theorem}

\newtheorem{claim}[thm]{Claim}

\newtheorem{lem}[thm]{Lemma}
\newtheorem*{lem*}{Lemma}
\newtheorem{mainthm}{Theorem}
\newtheorem*{mainthm*}{Theorem}
\newtheorem{maincor}[mainthm]{Corollary}
\newtheorem{prop}[thm]{Proposition}

\theoremstyle{definition}

\newtheorem*{case*}{Case}
\newtheorem{conj}[thm]{Conjecture}

\newtheorem{defn}[thm]{Definition}
\newtheorem*{defn*}{Definition}

\newtheorem*{exmp*}{Example}

\renewcommand{\thestep}{}

\theoremstyle{remark}

\newtheorem{case}{Case}\renewcommand{\thecase}{}

\newtheorem{rmk}[thm]{Remark}
\newtheorem*{rmk*}{Remark}

\makeatletter
\def\alphenumi{
  \def\theenumi{\alph{enumi}}
  \def\p@enumi{\theenumi}
  \def\labelenumi{(\@alph\c@enumi)}}
\makeatother

\makeatletter
\def\thecase{\@arabic\c@case}
\makeatother

\makeatletter
\def\thestep{\@arabic\c@step}
\makeatother

%% file: latexnewmacros.tex
\newcount\hh
\newcount\mm
\mm=\time
\hh=\time
\divide\hh by 60
\divide\mm by 60
\multiply\mm by 60
\mm=-\mm
\advance\mm by \time
\def\hhmm{\number\hh:\ifnum\mm<10{}0\fi\number\mm}

\let\oldmarginpar\marginpar
\renewcommand\marginpar[1]{\-\oldmarginpar[\raggedleft\footnotesize #1]%
{\raggedright\footnotesize #1}}

\newcommand\BB{\mathbb{B}}
\newcommand\CC{\mathbb{C}}

\newcommand\NN{\mathbb{N}}

\newcommand\RR{\mathbb{R}}

\newcommand\ZZ{\mathbb{Z}}

\newcommand\cE{{\mathcal{E}}}
\newcommand\cF{{\mathcal{F}}}

\newcommand\sB{{\mathscr{B}}}
\newcommand\sC{{\mathscr{C}}}
\newcommand\sD{{\mathscr{D}}}
\newcommand\sE{{\mathscr{E}}}
\newcommand\sF{{\mathscr{F}}}
\newcommand\sG{{\mathscr{G}}}
\newcommand\sH{{\mathscr{H}}}

\newcommand\sL{{\mathscr{L}}}
\newcommand\sM{{\mathscr{M}}}

\newcommand\sO{{\mathscr{O}}}

\newcommand\sU{{\mathscr{U}}}

\newcommand\sX{{\mathscr{X}}}
\newcommand\sY{{\mathscr{Y}}}

\newcommand\eps{\varepsilon}

\newcommand\less{\setminus}

\DeclareMathOperator{\Crit}{Crit}

\newcommand\diag{\operatorname{diag}}

\newcommand\divg{\operatorname{div}}

\newcommand\End{\operatorname{End}}

\newcommand\Fred{\operatorname{Fred}}

\newcommand\grad{\operatorname{grad}}

\newcommand\Hom{\operatorname{Hom}}

\newcommand\Ind{\operatorname{Index}}

\newcommand\Ker{\operatorname{Ker}}

\newcommand\vol{\operatorname{vol}}

\newcommand\id{{\mathrm{id}}}

\newcommand\mutatis{{\emph{mutatis mutandis }}}